\documentclass{amsart}
\usepackage{amssymb}
\usepackage{latexsym}
\usepackage{euscript}

\usepackage{todonotes} 
\usepackage{theoremref}

\usepackage{enumerate} 

\def\sD{{\mathfrak D}}   \def\sE{{\mathfrak E}}   \def\sF{{\mathfrak F}}
   \def\sH{{\mathfrak H}}   
   \def\sK{{\mathfrak K}}   \def\sL{{\mathfrak L}}
\def\sM{{\mathfrak M}}      
      \def\sR{{\mathfrak R}}

\def\st{{\mathfrak t}}

\def\sb{{\mathfrak b}}

      \def\dC{{\mathbb C}}

\def\cD{{\EuScript D}}

\def\wt#1{{{\widetilde #1} }}

\def\wh#1{{{\widehat #1} }}

\def\bm\chi{\mbox{\boldmath$\chi$}}
\def\half{{\frac{1}{2}}}

\def\RE{{\rm Re\,}}
\def\IM{{\rm Im\,}}
\def\ker{{\rm ker\,}}
\def\ran{{\rm ran\,}}
\def\cran{{\rm \overline{ran}\,}}
\def\dom{{\rm dom\,}}
\def\mul{{\rm mul\,}}

\def\cdom{{\rm d\overline{om}\,}}
\def\clos{{\rm clos\,}}

\let\xker=\ker \def\ker{{\xker\,}}

\def\uphar{{\upharpoonright\,}}

\newtheorem{theorem}{Theorem}[section]
\newtheorem{proposition}[theorem]{Proposition}
\newtheorem{corollary}[theorem]{Corollary}
\newtheorem{lemma}[theorem]{Lemma}

\theoremstyle{definition}
\newtheorem{example}[theorem]{Example}
\newtheorem{remark}[theorem]{Remark}

\numberwithin{equation}{section}

\begin{document}

\setcounter{page}{1}

\title[Factorized sectorial relations and form sums]
{Factorized sectorial relations, their maximal sectorial extensions, and form sums}

\author{S.~Hassi}
\author{A.~Sandovici}
\author{H.S.V.~de~Snoo}

\address{Department of Mathematics and Statistics \\
University of Vaasa \\
P.O. Box 700, 65101 Vaasa \\
Finland}
\email{sha@uwasa.fi}

\address{
Department of Mathematics and Informatics \\
"Gheorghe Asachi", Technical University of Ia\c{s}i \\
B-dul Carol I, nr. 11, 700506, Ia\c{s}i \\
Romania}
\email{adrian.sandovici@luminis.ro}

\address{Bernoulli Institute for Mathematics, Computer Science and Artificial Intelligence \\
University of Groningen \\
P.O. Box 407, 9700 AK Groningen \\
Nederland}
\email{hsvdesnoo@gmail.com}

%
%
%

\dedicatory{Dedicated to the memory of R.G. Douglas \\ with admiration for his contributions to mathematics}

\subjclass[2010]{Primary 47B44; Secondary 47A06, 47A07, 47B65.}

\keywords{Sectorial relation, Friedrichs extension, Kre\u{\i}n
extension, extremal extension, form sum.}


\begin{abstract}
In this paper sectorial operators, or more generally, sectorial relations and
their maximal sectorial extensions in a Hilbert space $\sH$ are considered.
The particular interest is in sectorial relations $S$,
which can be expressed in the factorized form
\[
 S=T^*(I+iB)T  \quad \text{or}\quad S=T(I+iB)T^*,
\]
where $B$ is a bounded selfadjoint operator in a Hilbert space $\sK$ and $T:\sH\to\sK$ or $T:\sK\to\sH$, respectively,
is a linear operator or a linear relation which is not assumed to be closed.
Using the specific factorized form of $S$, a description of all the maximal sectorial
extensions of $S$ is given with a straightforward construction of the extreme extensions
$S_F$, the Friedrichs extension, and $S_K$, the Kre\u{\i}n extension of $S$, which uses
the above factorized form of $S$. As an application of this construction the
form sum of maximal sectorial extensions of two sectorial relations is treated.
\end{abstract} \maketitle

\section{\textbf{Introduction}}\label{sec1}

Factorizations and decompositions of operators play a fundamental role in functional analysis and operator theory.
A well-known example is the ``Douglas lemma'' formulated in \cite[Theorem~1]{Dou} which makes a connection between
range inclusion, factorization, and ordering of operators. The importance of this connection is reflected by
the remarkable number of applications as well as its usage in the literature where this result plays a central role.
The present paper is not aimed to study factorizations on such a general level;
it is limited to unbounded nonnegative and sectorial operators, or more generally to sectorial relations $S$, which admit
a factorization of the form $S=T^*(I+iB)T$ or $S=T(I+iB)T^*$, where $T$ is a linear relation and $B\in \mathbf{B}(\sH)$ is selfadjoint. The main interest here is in the case where the (linear) relation $T$
is not closed and, therefore, $S$ need not be a maximal sectorial object. This leads to the extension problem for $S$.
Namely $H=T^*(I+iB)T^{**}$ or $H=T^{**}(I+iB)T^*$, respectively, is a maximal sectorial extension of $S$ and it is natural to ask
whether this $H$ is the only maximal sectorial extension of $S$.
However, since $T$ is not closed and no further conditions are required on $T$, the relation $S$ and its closure
can have positive defect. This yields immediately the problem ``what are the Friedrichs and the Kre\u{\i}n (maximal sectorial) extensions of $S$?''
In order to answer these questions some background definitions and facts on general sectorial operators and relations are first recalled.

A (linear) relation $S$ in a Hilbert space $\sH$ is said to be \textit{sectorial} with vertex at the origin
and semi-angle $\alpha$, $\alpha \in [0,\pi/2)$, if
\[
 | \IM (h',h) | \le (\tan \alpha) \RE (h',h),
 \quad  \{ h, h'\} \in H.
\]
Clearly, the closure of a sectorial relation is also sectorial.
A sectorial relation $S$ in a Hilbert space $\sH$ is said to be
\textit{maximal sectorial} if the existence of a sectorial relation
$\wt S$ in $\sH$ with $S \subset \wt S$ implies $\wt S = S$.
A maximal sectorial relation is automatically closed.

A sectorial relation $S$ generates a sectorial form, which in general is nondensely defined but closable
as stated in the next lemma; for a proof see \cite[Theorem~VI.1.27]{Kato}, \cite[Lemma~7.1]{HSSW17}.

\begin{lemma}\label{kak0}
Let $S$ be a sectorial relation in a Hilbert space $\sH$.
Then the form $\st_S$ given by
\[
 \st_S[\varphi,\psi]=(\varphi',\psi), \quad \{\varphi,\varphi'\}, \{\psi,\psi'\} \in S,
\]
with $\dom \st_S=\dom S$ is well-defined, sectorial, and closable.
\end{lemma}

According to the first representation theorem the closure of the form $\st_S$ determines a unique maximal sectorial
relation, which is the \textit{Friedrichs extension} $S_F$ of $S$; for the densely defined case see \cite[VI, Theorem 2.1]{Kato}
for the nondensely defined case see \cite{RB}, and for the linear relation case see \cite{Ar2,Ar}; a recent treatment in the general case
can be found in \cite[Section~7]{HSSW17}. The closure of the form $\st_{S}$ is denoted by $\st_{S_F}$.
According to the first representation theorem the domain of $S$ is a core for the closed form $\st_{S_F}$.
It is a consequence of the first representation theorem that there is a one-to-one correspondence between all
maximal sectorial relations $H$ in $\sH$ and all closed sectorial forms $\st$ (not necessarily densely defined) in $\sH$; cf. \cite[VI, Theorem 2.7]{Kato}, \cite[Theorem~4.3]{HSSW17}.
This correspondence is denoted by $\st\to H=:H_\st$; cf. Lemma \ref{kak0} when $S=H$ is maximal sectorial and $\st_H$ stands for the closure of $\st_S$.

All maximal sectorial relations $H$ admit a factorization which uses the real part $(\st_H)_{r}$
of the associated closed form $\st_H$. The real part is a closed nonnegative form
and by the first representation theorem there is a unique nonnegative selfadjoint relation $H_{r}$
corresponding to the closed nonnegative form $(\st_H)_{r}$.
The present formulation for the induced factorization for $H$ is taken from \cite[Theorem~6.2]{HSSW17},
for the densely defined case; see \cite[VI, Theorem 3.2]{Kato}.

\begin{lemma}\label{s-second}
Let $H$ be a maximal sectorial relation and let the closed
sectorial form $\st_H$ correspond to $H$.
Let $(\st_H)_r$ be the corresponding closed
nonnegative form and let $H_r$ be the corresponding nonnegative
selfadjoint relation. Then there exists a unique selfadjoint
operator $B \in \mathbf{B}(\sH)$, which is zero on
\begin{equation}\label{Bunique}
\sH \ominus \cran (H_{r})^{\half}_{\rm s}=\ker H_r\oplus\mul H_r,
\end{equation}
with $\| B \| = \tan \alpha$, such that the form $\st_H$ is given by
\[
 \st_H[h,k]=((I +iB) (H_{r})_{\rm s}^{\half} h, (H_{r})_{\rm s}^{\half} k),
 \quad
 h,k \in \dom \st_H =\dom H_{r}^{\half}.
\]
The maximal sectorial relation $H$ corresponding to $\st_H$ is given by
\begin{equation}\label{sss}
H =(H_{r})^{\half} (I + i B) (H_{r})^{\half}.
 \end{equation}
 The orthogonal operator part of $H$ is given by
 \begin{equation}\label{sss+}
H_{\rm s} = (H_{r})^{\half}_{\rm s} (I + i B) (H_{r})^{\half}_{\rm s},
  \end{equation}
where $(H_{r})_{\rm s}=P_{\cdom H} H_{r}$ is the operator part of $(H_{r})$.
\end{lemma}

It is the purpose of this paper to study properties of relations $S$
of the form $T^*(I+iB)T$ or $T(I+iB)T^*$ when $T$ is not assumed to be closed
and to apply these properties in the study of form sums and sums of sectorial relations.
In this case $S$ is sectorial, but typically it is not maximal sectorial.
By Lemma \ref{kak0} it induces, in general, a nondensely defined sectorial form,
which admits a closure that is again a sectorial form.
By the first representation theorem (see \cite{HSSW17}, \cite{Kato}) this closed
sectorial form corresponds to a maximal sectorial relation which, in addition, extends $S$.
This extension determines (the sectorial version of) the Friedrichs extension $S_F$ of $S$,
analogous to the case where $S$ is nonnegative. Since with $S$ also $S^{-1}$ is
sectorial (the sectorial version of) the Kre\u{\i}n extension of $S$ can be introduced as
$((S^{-1})_F)^{-1}$. The Friedrichs extension and the Kre\u{\i}n
extension are maximal sectorial extensions of $S$, which are in addition extremal. In the
nonnegative case all nonnegative selfadjoint extensions of $S$ are
between $S_{F}$ and $S_{K}$. In the sectorial case there is a
version of this property for their real parts (obtained via the
real part of the corresponding forms); see \cite[Theorem~7.6]{HSSW17} and  \cite[Theorem 3]{Ar} for a related result.

In Section \ref{sec2} some basic properties of sectorial relations of the form
\[
 S=T^*(I+iB)T \quad \text{and}\quad S'=T(I+iB)T^*
\]
are studied.
In particular, it is shown when the maximal sectorial extension
\[
 H=T^*(I+iB)T^{**}
\]
coincides with the Friedrichs extension $S_F$ of $S$ (Theorem \ref{SandS_F}) and when
\[
 H'=T^{**}(I+iB)T^*
\]
coincides with the Kre\u{\i}n extension $(S')_K$ of $S'$ (Theorem~\ref{SandS_K}).
To give a complete picture of the situation the case $S=T^*(I+iB)T$ is investigated in detail in Section \ref{sec2.2} by
giving a general procedure that leads to the description of the Friedrichs extension $S_F$ and the Kre\u{\i}n extension $S_K$ of $S$
and, in fact, all the extremal extensions of $S$ combined with their associated closed sectorial forms; see Theorem~\ref{s-F-N} and Proposition~\ref{Msecforms}.

In Section \ref{sec3} a particular case of a sectorial relation with the factorization $S'=T(I+iB)T^*$ is investigated.
The choice for $S'$ treated here occurs when studying the form sums $\st_1+\st_2$ of two closed sectorial (in particular nonnegative) forms
in a Hilbert space $\sH$. To explain this let $H_1$ and $H_2$ be the maximal sectorial relations in $\sH$ associated with $\st_1$ and $\st_2$, respectively.
Since the sum $\st_1+\st_2$ is a closed form in $\sH$, there is again an associated maximal sectorial relation $\wh H$ that corresponds to $\st_1+\st_2$;
cf. \cite[Chapter VI]{Kato}.
In a natural way $\wh H$ can be seen as a maximal sectorial extension of the operator-like sum $H_1+H_2$ of the maximal sectorial relations $H_1$ and $H_2$;
for this reason $\wh H$ is called the \textit{form sum extension} of $H_1+H_2$.
To investigate the form sum extension $\wh H$ of $H_1+H_2$ the Friedrichs and the Kre\u{\i}n extension of the sum $H_1+H_2$ will be constructed;
see Theorems~\ref{ss-twee} and~\ref{KVNext}. This leads to a description of all maximal sectorial extensions that are extremal in Proposition \ref{s-sum-caracter}.
It turns out that the form sum extension $\wh H$ of $H_1+H_2$ need not be extremal; a characterization for this is given in Theorem \ref{s-sum-een}.

For the treatment in Section \ref{sec3} the factorized form of $H_1+H_2$ is again playing a key role.
Indeed, according to Lemma \ref{s-second} $H_1$ and $H_2$ as maximal sectorial relations admit the factorizations
\[
  H_{j} = A_{j}^{\half} (I+iB_{j}) A_{j}^{\half},
\]
where $A_{j}$ (the real part of $H_{j}$), $j=1,2$, are
nonnegative selfadjoint relations in $\sH$ and $B_{j}$, $j=1,2$,
are bounded selfadjoint operators in $\sH$. This yields the following factorization of $H_1+H_2$:
\[
 H_1+H_2=A_1^{\half}(I+iB_1)A_1^{\half} + A_2^{\half}(I+iB_1)A_2^{\half}= \Phi \left(I_{\sH^2}+i (B_1\oplus B_2)\right) \Phi^*,
\]
where $\Phi$ stands for the row operator (or relation) from $\sH\times \sH$ to $\sH$ formally defined by
\[
 \Phi=\begin{pmatrix}
        A_1^{\half} & A_2^{\half}
      \end{pmatrix}
\]
and whose adjoint $\Phi^*$ is the column operator (or relation) formally given by
\[
 \Phi^*=\begin{pmatrix}
        A_1^{\half} \\ A_2^{\half}
      \end{pmatrix}: \sH\to \sH\times \sH.
\]
Hence $H_1+H_2$ is a sectorial relation which admits a factorization of the form $S=T(I+iB)T^*$
with $T=\Phi$ and $B=B_1\oplus B_2$. Even in the case that $H_1$ and $H_2$ are densely defined operators,
the operator $T$ is typically neither closed nor closable; it can even be singular (cf. \cite{HSeSn2018})
if for instance $\dom A_1^\half\cap \dom A_2^\half=\{0\}$.

For some general developments on the notions of Friedrichs and Kre\u{\i}n extensions the reader is referred to see
\cite{AN,C,CS,H,HMS,Kato,Kr1} in the case of nonnegative operators and relations and \cite{Ar2,HSSW17,Kato,RB}
in the case of sectorial relations. Treatments of extremal extensions can be found in \cite{Ar,AHSS,HSSW2}, while
construction of factorizations for these extensions have been treated in \cite{AHSS,HSSW2,PS,SSi,SS,ST12} and
the notion of form sums appears in \cite{FM,HSSW1,HSSW3,ST12}.
Throughout this paper \cite{HSSW17} will be used as a standard reference for various concepts and results on sectorial relations and
their extensions; therein one can also find a more detailed description on the literature and developments in this area.
As another general overview on sectorial relations we would like to mention the survey paper of Yu.M.~Arlinski\u{\i} \cite{Ar12}.

Finally it should be pointed out that the results in Section \ref{sec2} apply in particular to the factorized nonnegative relations
of the form
\[
 S=T^*T \quad \text{or} \quad S=TT^*,
\]
where $T$ is a linear relation or operator which is not assumed to be closed.
The special case where $S=T^*T$ is a densely defined nonnegative operator and the densely defined operator
$T$ is not closed has been recently investigated in \cite{ST12}.
Similarly, the results in Section \ref{sec3} extend the earlier results concerning
the sum of nonnegative relations obtained in \cite{HSSW1} and \cite{HSSW3}.

\section{\textbf{Some characteristic properties of $T^*(I+iB)T$ and $T(I+iB)T^*$}}\label{sec2}

In this section the class of linear relations $S$ in a Hilbert space $\sH$ which admit a factorization of the form
\begin{equation}\label{Sfactored}
 S=T^*(I+iB)T \quad \textrm{or} \quad S=T(I+iB)T^*
\end{equation}
will be studied; here $B$ is a bounded operator in a Hilbert space $\sK$ and $T$ is a linear operator or a linear relation (not necessarily closed)
from  $\sH$ to $\sK$ or from $\sK$ to $\sH$, respectively.
This class contains all densely defined, not necessarily closed, sectorial relations, but also a wide class of multivalued sectorial
relations; for instance Lemma \ref{s-second} shows that all maximal sectorial relations admit a factorization of the form \eqref{Sfactored}
with $T$ a closed operator or a closed relation; see \eqref{sss}, \eqref{sss+}. Conversely, if $T$ is closed then the relation
$S$ in \eqref{Sfactored} is maximal sectorial. In the case that $T$ is not closed the relation $S$ need not be maximal sectorial,
but it has maximal sectorial extensions.

\subsection{Some basic properties of $T^*CT$}
To study operators and relations $S$ determined by the factorization  \eqref{Sfactored},
the following observations concerning products of the form $T^*CT$ are helpful.

\begin{lemma}\label{lem2.1}
Let $T$ be a relation from a Hilbert space $\sH$ to a Hilbert space $\sK$, let $C\in\mathbf{B}(\sK)$ and let
the linear relation $W$ in $\sH$ be defined as the product
\[
 W=T^*CT.
\]
Then the following statements hold:
\begin{enumerate}\def\labelenumi{\rm (\roman{enumi})}
\item If $C$ has the property
\begin{equation}\label{C1}
 (Cf,f)=0 \quad \Rightarrow \quad f=0
\end{equation}
then for each $\varphi' \in \ran W$ there is precisely one $\alpha \in \sK$
such that for any $\varphi \in \sH$ with $\{\varphi, \varphi'\} \in W$ one has
 \begin{equation}\label{einz}
 \{\varphi, \alpha\} \in T \quad \mbox{and} \quad \{C\alpha, \varphi'\} \in T^{*},
\end{equation}
in which case
\begin{equation}\label{zwei}
(\varphi',\varphi)=(C\alpha, \alpha).
\end{equation}
Moreover, for every $\{\varphi, \varphi'\} \in W$ the element $\varphi \in
\sH$ is uniquely determined modulo $\ker T$.
In particular, $W$ satisfies the following identities
\begin{equation}\label{mulker}
 \mul W=\mul T^* \;\text{ and } \;  \ker W=\ker T.
\end{equation}

\item If for any sequence $(f_n)$ the operator $C$ satisfies the property
\begin{equation}\label{C2}
 \lim_{n\to\infty}(Cf_n,f_n) = 0 \quad \Rightarrow \quad \lim_{n\to \infty} f_n=0,
\end{equation}
then the following implication is also true
\[
 T \quad \text{is closed}\quad \Rightarrow \quad W \quad \text{is closed}.
\]
In particular, the closure of $W$ satisfies $W^{**}\subset T^*CT^{**}$ and
\[
 \mul W^{**}=\mul W=\mul T^*, \quad \ker T\subset\ker W^{**}\subset \ker T^{**}.
\]
\end{enumerate}
\end{lemma}

\begin{proof}
(i) Let $\varphi' \in \ran W$. Then for any $\varphi \in \sH$ such that
$\{\varphi, \varphi'\} \in W$ there exists $\alpha \in \sK$ such that
\eqref{einz} holds and consequently \eqref{zwei} is satisfied, too.
To see the uniqueness properties of $\alpha$ and $\varphi$ assume that also $\{\varphi_{0},\varphi'\}  \in W$ with $\varphi_0 \in \sH$.
Then analogously there exists an element $\alpha_{0} \in \sK$ such that
 \[
  \{\varphi_{0}, \alpha_{0} \} \in T,
 \quad \{C\alpha_{0}, \varphi'\} \in T^{*},
\]
which via \eqref{einz} leads to
\[
\{\varphi-\varphi_{0}, \alpha-\alpha_{0} \} \in T, \quad
 \{C(\alpha-\alpha_{0}), 0\} \in T^{*}.
\]
Hence $(C(\alpha-\alpha_{0}), \alpha-\alpha_{0})=0$ and now the assumption in (i)
implies that $\alpha=\alpha_{0}$, i.e., $\alpha$ is unique.
Moreover, one concludes that $\{\varphi-\varphi_{0}, 0\} \in T$, which proves the claimed uniqueness of $\varphi$
and the equality $\ker W=\ker T$.

To see that $\mul W=\mul T^{*}$, assume that $\{0, \varphi'\} \in W$.
Then it follows from \eqref{einz} and \eqref{zwei} that $\alpha=0$,
which implies that $\mul W \subset \mul T^*$. The reverse inclusion is trivial and hence \eqref{mulker} is shown.

(ii) Assume that $T$ is closed. To see that $W$ is closed,
let $\{\varphi_n,\varphi_n'\} \in W$ converge to $\{\varphi,\varphi'\} \in \sH$.
Then there exists a sequence of vectors $\alpha_n\in \sK$ such that
\begin{equation*}
 \{\varphi_n, \alpha_n\} \in T \quad \mbox{and} \quad \{C \alpha_n, \varphi_n'\} \in T^{*},
\end{equation*}
and it follows that
\begin{equation*}
(C \alpha_n, \alpha_n)=(\varphi_n',\varphi_n)\to (\varphi', \varphi).
\end{equation*}
Consequently,
\[
 (C (\alpha_n-\alpha_m, \alpha_n-\alpha_m) \to 0, \quad n,m\to \infty,
\]
and now the assumption in (ii) shows that $(\alpha_n)$ is a Cauchy sequence in $\sK$.
Hence, $\alpha_n$ converges to some $\alpha$ in $\sK$ and
one concludes that $\{\varphi, \alpha\} \in T$ and $\{C\alpha, \varphi'\} \in T^{*}$.
Thus $\{\varphi, \varphi'\}\in W$ and $W$ is closed.

Finally, the inclusion $W\subset T^*CT^{**}$ is clearly true and since $T^{**}$ is closed,
also $T^*CT^{**}$ is closed by the property \eqref{C2}. Therefore,
\[
 W\subset W^{**}\subset T^*CT^{**}.
\]
By the statement (i) this leads to $\ker T\subset \ker W^{**}\subset \ker T^*CT^{**}=\ker T^{**}$ and
\[
 \mul T^* = \mul W \subset \mul W^{**}\subset \mul T^*CT^{**}=\mul T^*,
\]
so that $\mul W=\mul W^{**}=\mul T^*$. This completes the proof.
 \end{proof}

By changing the roles of $T$ and $T^*$ in Lemma~\ref{lem2.1} leads to the following result.

\begin{corollary}\label{cor2.2}
Let $T$ be a relation from a Hilbert space $\sK$ to a Hilbert space $\sH$, let $C\in\mathbf{B}(\sK)$ and let
the linear relation $W$ in $\sH$ be defined as the product
\[
 W=TCT^*.
\]
Then:
\begin{enumerate}\def\labelenumi{\rm (\roman{enumi})}
\item the assumption \eqref{C1} implies that $W$ satisfies the properties in part (i) in Lemma \ref{lem2.1}
with the roles of $T$ and $T^*$ interchanged.

\item If \eqref{C2} holds, then $W^{**}\subset T^{**}CT^*$ and if $T$ is closed then also $W$ is closed. Moreover,
\[
 \ker W^{**}=\ker W=\ker T^*, \quad \mul T\subset\mul W^{**}\subset \mul T^{**}.
\]
\end{enumerate}
\end{corollary}
\begin{proof}
(i) This assertion is proved by interchanging the roles of $T$ and $T^*$ in the proof of Lemma~\ref{lem2.1}.

(ii) The statement with $T$ closed is obtained by applying part (ii) of Lemma~\ref{lem2.1} to $T^*$ instead of $T$.
As to the remaining assertions observe that $W\subset W^{**}\subset T^{**}CT^*$ and hence
$\mul T=\mul W\subset \mul W^{**}\subset \mul T^{**}CT^*=\mul T^{**}$. Moreover,
\[
 \ker T^* = \ker W \subset \ker W^{**}\subset \ker T^*CT^{**}=\ker T^*,
\]
and thus $\ker W=\ker K^{**}=\ker T^*$.
\end{proof}

In particular, all (positively or negatively) definite operators $C$ satisfy the assumption (i) in Lemma \ref{lem2.1}
and all uniformly definite operators $C$ satisfy the assumption (ii) in Lemma \ref{lem2.1}. Of course there are many
other operators where assumption (i) or (ii) in Lemma \ref{lem2.1} is satisfied.
Notice that if $C$ satisfies the assumption (i) or (ii) in Lemma \ref{lem2.1}, then the same is true also for the following operators
\[
 C^*;\quad \eta\, C\,\,(0\neq \eta \in \dC); \quad X^*CX,
\]
where $X$ is a bounded operator with bounded inverse.
In the present paper Lemma \ref{lem2.1} is applied to a special class of sectorial relations.

\begin{proposition}\label{SectorialCor}
Let $T$ be a linear relation and let $C=I+iB$ for some selfadjoint operator $B\in\mathbf{B}(\sK)$.
Then
\[
 S=T^{*}(I+iB)T \quad \text{and} \quad  S'=T(I+iB)T^*,
\]
with $T$ from $\sH$ to $\sK$ or from $\sK$ to $\sH$, respectively, are sectorial relations in $\sH$ with
vertex at the origin and semi-angle at most
$\arctan \|B\|$, and $S$ admits the properties (i) and (ii) in Lemma \ref{lem2.1} while $S'$ admits the properties in
Corollary \ref{cor2.2}.

If, in addition, the relation $T$ is closed, i.e. $T=T^{**}$, then $S$ and $S'$ as well as their adjoints are maximal sectorial with
\[
 S^*=T^*(I-iB)T, \quad (S')^*=T(I-iB)T^*.
\]
\end{proposition}
\begin{proof}
Since $B$ is selfadjoint one concludes that for all $\{\varphi, \varphi'\} \in S$:
\[
 | \IM (\varphi',\varphi) |=|(B \alpha, \alpha)|
 \leq \|B\| (\alpha, \alpha)=\|B\| \RE (\varphi',\varphi);
\]
cf. the beginning of the proof of Lemma \ref{lem2.1}.
Hence $S$ is sectorial with vertex at the origin and semi-angle at most
$\arctan \|B\|$. The argument concerning $S'$ remains the same.

The properties for $S$ in Lemma \ref{lem2.1} and for $S'$ in Corollary \ref{cor2.2} follow from that fact
that the real part of $C=I+iB$ as the identity operator is boundedly invertible.

Finally, if $T$ is closed then also $S=T^*(I+iB)T$ and $S'=T(I+iB)T^*$ are closed by Lemma \ref{lem2.1}.
The fact that $S$, $S'$ are maximal sectorial can be found in \cite{HS2019}.
Then also their adjoints are maximal sectorial and since
$S^*=(T^*(I+iB)T)^*\supset T^*(I-iB)T$, where $T^*(I-iB)T$ is maximal sectorial (again see \cite{HS2019}),
equality $S^*=T^*(I-iB)T$ prevails. The equality $(S')^*=T(I-iB)T^*$ is now obtained by changing the roles of $T$ and $T^*$.
\end{proof}

It is a consequence of Lemma \ref{s-second} that a set $\cD$ is a core for the form $\st_H$ precisely
when $\cD$ is a core for its real part $(\st_H)_r$.
This observation combined with Lemmas~\ref{kak0}, \ref{s-second}, and \ref{lem2.1} leads to
a characterization concerning the factorization \eqref{Sfactored} of $S$ and its Friedrichs extension $S_F$.

\begin{theorem}\label{SandS_F}
Let $S$ be a not necessarily closed sectorial relation in the Hilbert space $\sH$.
Then the following assertions are equivalent:
\begin{enumerate}\def\labelenumi{\rm (\roman{enumi})}
\item $\mul S=\mul S^{*}$;
\item  there exists a Hilbert space $\sK$, a linear relation $T:\sH\to \sK$ with $\dom T=\dom S$ and a selfadjoint operator $B\in \mathbf{B}(\sK)$,
such that
\begin{equation}\label{sfactor}
 S=T^*(I+iB)T    \; \text{ and }\; S_F=T^*(I+iB)T^{**}.
\end{equation}
\end{enumerate}
Moreover, in (ii) $T:\sH\to \sK$ can be assumed to be a closable operator.
\end{theorem}
\begin{proof}
(i) $\Rightarrow$ (ii) Assume that $S$ is a sectorial relation such that $\mul S=\mul S^*$.
Let $S_F$ be the Friedrichs extension of $S$ associated with the closure of the form $\st_S$ defined in Lemma~\ref{kak0}.
By Lemma \ref{s-second} $S_F$ admits the factorization \eqref{sss} with $(S_F)_r^{\half}$ and $B\in \mathbf{B}(\sH)$,
while its operator part is factorized as in \eqref{sss+} using the operator part of $(S_F)_r^{\half}$.
Now introduce the operator $T$ as the following restriction:
\begin{equation}\label{Tconstr}
 T:=((S_F)_r^{\half})_{\rm s}\upharpoonright \dom S.
\end{equation}
Recall that $\dom S$ is a core for the forms $\st_{S_F}$ and $(\st_{S_F})_r$. Consequently, $\dom S$ is also a core
for the operator part, i.e., $\clos T=((S_F)_r^{\half})_{\rm s}$. In particular, $T$ is closable.
Moreover,
\begin{equation}\label{Tconstr*}
 T^*=(((S_F)_r^{\half})_{\rm s})^*=(S_F)_r^{\half},
\end{equation}
where the adjoint is taken in $\sH$; notice that $(\dom T)^\perp=\mul S_F=\mul (S_F)_r=\mul (S_F)_r^{\half}$.

We claim that $S=T^*(I+iB)T$. In fact, by the definition of $T$ one has $(\dom S)^\perp=(\dom T)^\perp=\mul T^*$ and hence
the assumption $\mul S=\mul S^*$ yields
\[
 \mul S=\mul T^*=\mul S_F.
\]
This identity combined with the inclusion $S\subset S_F$ and the identities \eqref{Tconstr} and \eqref{Tconstr*} shows
that
\[
 S=\{ \{f,f'\}\in S_F: f\in \dom S\} = T^*(I+iB) T.
\]

(ii) $\Rightarrow$ (i)
By Proposition \ref{SectorialCor} every relation $S$ of the form \eqref{sfactor} is sectorial.
Clearly,
\[
 S\subset S^{**}\subset T^*(I+iB)T^{**},
\]
and by the assumption $S_F=T^*(I+iB)T^{**}$. Since the domain of $S$ is a core for the closed form $\st_{S_F}$,
one has $\mul S_F=\mul S^*$. On the other hand, by Lemma \ref{lem2.1}~(i) (cf. Proposition \ref{SectorialCor}) $S$ and $S_F$ in
\eqref{sfactor} satisfy $\mul S=\mul T^*$ and $\mul S_F=\mul T^*$. Therefore, $\mul S=\mul S^*$ holds.

The last assertion is clear from the proof (i) $\Rightarrow$ (ii).
\end{proof}

In the case that $S$ is densely defined Theorem~\ref{SandS_F} gives the following result.

\begin{corollary}\label{SebTarCor}
Let $S$ be a densely defined sectorial operator in the Hilbert space $\sH$.
Then there exists a Hilbert space $\sK$, a closable operator $T:\sH\to \sK$ with $\dom T=\dom S$ and a selfadjoint operator $B\in \mathbf{B}(\sK)$,
such that
\[
 S=T^*(I+iB)T    \; \text{ and }\; S_F=T^*(I+iB)T^{**}.
\]
\end{corollary}
\begin{proof}
If $S$ is densely defined, then $\mul S\subset \mul S^*=(\dom S)^{\perp}=\{0\}$ and now the statement follows from Theorem~\ref{SandS_F}.
\end{proof}

Corollary  \ref{SebTarCor} extends \cite[Theorem 5.3]{ST12}:
if $S\geq 0$ is a densely defined operator then there is a closable operator $T$ in $\sH$ such that
\[
 S=T^*T  \; \text{ and }\; S_F=T^*T^{**};
\]
in \cite{ST12} these factorizations for $S\geq 0$ were constructed in another way.

Theorem~\ref{SandS_F} involves the Friedrichs extension $S_F$ of $S$. There is a similar result for the
Kre\u{\i}n extension $S_K$ of $S$. The Kre\u{\i}n extension in the nonnegative case was introduced and studied in \cite{Kr1}.
Following the approach used in the nonnegative case in \cite{AN,CS} this extension is defined for a sectorial relation $S$
using the inverse $S^{-1}$ by the formula
\[
 S_K=((S^{-1})_F)^{-1};
\]
cf. \cite[Definition 2]{Ar}, \cite[Definition 7.4]{HSSW17}. This leads to the following analog of Theorem~\ref{SandS_F}.

\begin{theorem}\label{SandS_K}
Let $S$ be a not necessarily closed sectorial relation in the Hilbert space $\sH$.
Then the following assertions are equivalent:
\begin{enumerate}\def\labelenumi{\rm (\roman{enumi})}
\item $\ker S=\ker S^{*}$;
\item  there exists a Hilbert space $\sK$, a linear relation $T:\sK\to \sH$ with $\ran T=\ran S$ and a selfadjoint operator $B\in \mathbf{B}(\sK)$,
such that
\begin{equation}\label{sfactorSK}
 S=T(I+iB)T^*    \; \text{ and }\; S_K=T^{**}(I+iB)T^*.
\end{equation}
\end{enumerate}
Moreover, in (ii) the inverse $T^{-1}:\sH\to \sK$ can be assumed to be a closable operator.
\end{theorem}
\begin{proof}
(i) $\Rightarrow$ (ii) Assume that $S$ is a sectorial relation such that $\ker S=\ker S^*$ and consider its inverse
$S^{-1}$. By the assumption one has $\mul S^{-1}=\mul (S^{-1})^*$ and hence by Theorem~\ref{SandS_F} there exist a linear relation
$\wt T:\sH\to \sK$, which can be assume to be closable, and a selfadjoint operator $\wt B\in\mathbf{B}(\sK)$ such that
\[
 S^{-1}=\wt T^*(I+i\wt B) \wt T, \quad (S^{-1})_F=\wt T^*(I+i\wt B) \wt T^{**}.
\]
Passing to the inverses one obtains
\[
 S=\wt T^{-1}(I+i\wt B)^{-1} (\wt T^*)^{-1}, \quad S_K=(\wt T^{**})^{-1}(I+i\wt B)^{-1} (\wt T^*)^{-1}.
\]
Since $(I+i\wt B)^{-1}=(I+\wt B^2)^{-\half} (I-i\wt B)  (I+\wt B^2)^{-\half}$, this yields
\[
 S=T(I-i\wt B)T^*, \quad S_K=T^{**}(I-i\wt B) T^*,
\]
where $T=\wt T^{-1}(I+\wt B^2)^{-\half}$ and $T^*=(I+\wt B^2)^{-\half}(\wt T^{-1})^*$; note that $(\wt T^{-1})^*=(\wt T^*)^{-1}$.
By construction $\ran T=\dom \wt T=\dom S^{-1}=\ran S$. Since $(I+\wt B^2)^{\half}$ is bounded with bounded inverse
one has $\clos T^{-1}=(I+\wt B^2)^{\half}(\clos \wt T)$ and thus $T^{-1}$ is closable precisely when $\wt T$ is closable.
Therefore the assertions in (ii) hold and one has the factorizations \eqref{sfactorSK} with $T=\wt T^{-1}(I+\wt B^2)^{-\half}$ and $B=-\wt B$.

(ii) $\Rightarrow$ (i)
By Proposition \ref{SectorialCor} every relation $S$ of the form \eqref{sfactorSK} is sectorial.
Clearly,
\[
 S\subset S^{**}\subset T^{**}(I+iB)T^*,
\]
and by the assumption $S_K=T^{**}(I+iB)T^*$. Since the range of $S$ is a core for the closed form $\st_{(S^{-1})_F}$,
one has $\ker S_K=\ker S^*$. On the other hand, by Proposition~\ref{SectorialCor} (or Corollary \ref{cor2.2}) $S$ and $S_K$ in
\eqref{sfactorSK} satisfy $\ker S=\ker T^*$ and $\ker S_K=\ker T^*$. Therefore, $\ker S=\ker S^*$ holds.
\end{proof}

It is clear that there is an analog of Corollary \ref{SebTarCor} concerning the factorization $T(I+iB)T^*$ whose formulation is left to the reader.
In what follows the purpose is to offer a construction for maximal sectorial extensions, in particular, for the Friedrichs extension and the Kre\u{\i}n extension,
for sectorial relations $S$ and $S'$ which admit a factorization as in Proposition \ref{SectorialCor} without any additional conditions
as in Theorems~\ref{SandS_F} and~\ref{SandS_K}. In the next section attention is limited to the case $S=T^*(I+iB)T$.
On the other hand, in Section \ref{sec3} a special case where $S$ admits a factorization $S=T(I+iB)T^*$ is treated
by investigating the form sum of two maximal sectorial relations.

\subsection{Maximal sectorial extensions of $T^{*}(I+iB)T$ with nonclosed $T$}\label{sec2.2}

In Lemma \ref{lem2.1} it has been shown that the relation
$T^{*}(I+iB)T$, when $T$ is not necessarily closed, is still
sectorial. The purpose in this section is to show that
$T^{*}(I+iB)T$ has maximal sectorial extensions and, in particular,
to describe all of them. It is clear that every maximal sectorial
extension of $T^{*}(I+iB)T$ is also an extension of the closure
$\clos(T^{*}(I+iB)T)$. On the other hand,
\begin{equation}\label{4eq1}
\clos(T^{*}(I+iB)T) \subset T^{*}(I+iB)T^{**},
\end{equation}
since by Proposition~\ref{SectorialCor} the relation $T^{*}(I+iB)T^{**}$ is
closed and, in fact, a maximal sectorial relation in $\sH$. Hence,
it is clear that without any additional assumptions on
$T^{*}(I+iB)T$ the relation on the right-hand side of \eqref{4eq1}
is one of the maximal sectorial extensions of $S:=T^{*}(I+iB)T$.
Under the additional condition $\mul S=\mul S^*$
one has $S_F=T^{*}(I+iB)T^{**}$; see Theorem \ref{SandS_F}.
In what follows this additional condition will not be assumed.

The aim now is to describe all extremal maximal sectorial
extensions of $S=T^{*}(I+iB)T$, including the Friedrichs extension $S_F$,
using the given factorized form of $S$.
The purpose is to incorporate explicitly the prescribed structure of
$S=T^{*}(I+iB)T$ in the construction of maximal sectorial
extensions of $T^{*}(I+iB)T$. The approach presented here
has the advantage that it prevents the construction of an auxiliary Hilbert space
when compared with the procedure appearing in \cite{HSSW17} for
a sectorial relations $S$ without additional information on its structure.

Recall from Lemma~\ref{lem2.1} that for each $\varphi',
\psi' \in \ran S$ there exist unique elements $\alpha, \beta \in
\sK$ with
\begin{equation}\label{4eq2}
 \{\varphi, \alpha\} \in T, \quad \{ (I+iB) \alpha , \varphi'\} \in T^{*},
 \quad
 \{\psi, \beta\} \in T, \quad \{ (I+iB) \beta , \psi'\} \in T^{*}.
\end{equation}
Next introduce the linear subspace $\sM_0$ of the Hilbert space
$\sK$ via
\begin{equation}\label{M0Def}
 \sM_0=\{ \alpha \in \sK: \, \alpha \in \ran T, (I+iB)\alpha \in \dom T^* \},
\end{equation}
and let $\sM$ be the closure of $\sM_0$ in $\sK$. Moreover, let
$B_m$ be the compression of $B$ to $\sM$:
\begin{equation}\label{BmDef}
 B_m:=P_\sM B\uphar \sM\in \mathbf{B}(\sM).
\end{equation}
Then $B_m$ is a selfadjoint operator in $\sM$.
Next we construct a pair of relations $Q\subset T$ and
$J\subset Q^*$, which will be used to describe the minimal and
maximal and, in fact, all extremal maximal sectorial extensions of
$T^*(I+iB)T$.

\begin{lemma}\label{lemma4a}
Associate with $T^*(I+iB)T$ the subspace $\sM_0$ of $\sK$ in
\eqref{M0Def} and the compression $B_m$ in \eqref{BmDef} and define
the linear relation $Q$ from $\sH$ to $\sM$ and the linear relation
$J$ from $\sM$ to $\sH$ via
\[
\begin{array}{l}
 Q=\{ \{\varphi, \alpha\} \in T:\, \alpha\in \sM_0 \},  \\[3pt]
 J=\{ \{ (I+iB_m)\alpha, \varphi'\}:\,  \alpha\in \sM_0,\; \{ (I+iB)\alpha ,\varphi'\}\in T^* \}.
\end{array}
\]
Then $Q \subset J^{*}$, or equivalently, $J\subset Q^{*}$, and $Q$
is a closable operator with dense range in $\sM$, while $J$ is
densely defined and satisfies $\mul J=\mul J^{**}=\mul T^*$.
Moreover, one has the equality
\[
 T^*(I+iB)T=J(I+iB_m)Q.
\]
\end{lemma}
\begin{proof}
It is first shown that $Q\subset J^*$. For this let $\{\varphi,
\alpha\}\in Q$ and $\{(I+iB_m)\beta, \psi'\}\in J$. Then
$\alpha,\beta\in \sM_0$ and they correspond to some
$\{\varphi,\varphi'\},\{\psi,\psi'\}\in T^*(I+iB)T$ via
\eqref{4eq2}. In particular, $\{\varphi,\alpha\}\in T$ and hence
\[
 (\psi',\varphi)-((I+iB_m)\beta,\alpha)=(\psi',\varphi)-((I+iB)\beta,\alpha)
 =0,
\]
where the last equality follows from \eqref{4eq2}. Hence $Q \subset
J^{*}$ and, equivalently, $J\subset Q^{*}$.

Next it is shown that the set $(I+iB_m)(\sM_0)$ is dense in $\sM$.
Assume conversely that there exists $\beta\in\sM$ such that
$((I+iB_m)\alpha,\beta)=0$ for all $\alpha\in\sM_0$. Let
$\alpha_n\in\sM_0$ be a sequence such that $\alpha_n\to \beta$ (in
$\sK$). Then
\[
 0=((I+iB_m)\alpha_n,\beta)=((I+iB)\alpha_n,\beta)
\]
and by taking limit this leads to
\[
 0= \lim_{n\to\infty}((I+iB)\alpha_n,\beta)=((I+iB)\beta,\beta),
\]
which implies that $\beta=0$. Consequently, $J$ is densely defined
in $\sM$ and hence its adjoint $J^*$ is an operator. Since $Q\subset
J^*$, the relation $Q$ is a closable operator. Furthermore, by
definition, $\ran Q$ is dense in $\sM$.

Now consider the multivalued parts of $J$ and its closure $J^{**}$.
The inclusion $\mul T^*\subset \mul J$ follows from
the definition of $J$ and clearly $\mul J \subset \mul J^{**}$.
On the other hand, if $\psi'\in \mul J^{**}$ then there are sequences
$\{\psi_n,\beta_n\}\in T$ and $\{(I+iB)\beta_n,\psi_n'\}\in T^*$ such that
$(I+iB_m)\beta_n\to 0$ and $\psi_n'\to \psi'$. Then necessarily $\beta_n\to 0$ in $\sM$
since $B_m$ is selfadjoint and hence $(I+iB_m)$ is boundedly invertible in $\sM$.
Then $(I+iB)\beta_n\to 0$ in $\sK$ and consequently $\{0,\psi'\}\in T^*$, i.e.
$\psi'\in \mul T^*$. Hence, $\mul J^{**}\subset \mul T^*$ and the equalities
$\mul J=\mul J^{**}=\mul T^*$ follow.

Finally, the last identity is shown. The inclusion
$T^*(I+iB)T\subset J(I+iB_m)Q$ follows directly from \eqref{4eq2}
and the definitions of $Q$ and $J$. The reverse inclusion
$J(I+iB_m)Q\subset T^*(I+iB)T$ is clear from the definitions of
$Q$ and $J$.
\end{proof}

It follows from Lemma~\ref{lemma4a} that $J^*$ is a closed operator from $\sH$ into $\sM$
and its domain is dense in $(\mul J^{**})^\perp=\cdom T$. Moreover, by definition the domain of
the restriction $Q\subset J^*$ is given by $\dom Q=\dom (T^*(I+iB)T)$; cf. \eqref{M0Def}.
The next result characterizes a class of closed sectorial forms generated by linear
operators $K$ lying between these two operators.

\begin{proposition}\label{Msecforms}
Let the notation be as in Lemma~\ref{lemma4a} and let $K$ be a
linear operator satisfying
\[
Q \subset K \subset J^{*}.
\]
Then the form induced by $K$:
\[
 \st_K[h,k]=\langle(I+iB_m)Kh,Kk\rangle, \quad h,k \in \dom K,
\]
is closable. The closure of the form $\st$ is given by
\begin{equation}\label{qrj--}
 \st_{K^{**}}[h,k]=\langle(I+iB_m)K^{**}h,K^{**}k\rangle,
 \quad h, \, k \in \dom K^{\ast \ast},
\end{equation}
and the corresponding maximal sectorial relation $K^*(I+iB_m)K^{**}$
is an extension of the sectorial relation $T^*(I+iB)T$.
\end{proposition}

\begin{proof}
Clearly $K$ is closable and its closure $K^{**}$ satisfies
\[
 Q \subset K \subset K^{**} \subset J^{*}, \quad J \subset J^{**} \subset K^*.
\]
Hence, the form $\st_K$ is also closable and its closure is
determined by $K^{**}$ as in \eqref{qrj--}. By Proposition~\ref{SectorialCor}
$K^*(I+iB_m)K^{**}$ is maximal sectorial and it clearly corresponds to the closed form
$\st_{K^{**}}$ in \eqref{qrj--}; cf. Lemma \ref{s-second}.
Furthermore, since $J\subset K^{*}$ and
$Q\subset K^{**}$ it follows from Lemma~\ref{lemma4a} that
\[
 T^*(I+iB)T=J(I+iB_m)Q\subset K^*(I+iB_m)K^{**},
\]
which proves the last statement.
\end{proof}

It is clear from Proposition~\ref{Msecforms} that
\[
 K_1\subset K_2 \quad \Longleftrightarrow \quad \st_{K_1} \subset
 \st_{K_2}
\]
and that these forms are closed precisely when the operators $K_1$
and $K_2$ are closed. The next result shows that the minimal choice
$K_1=Q^{**}$ in fact corresponds to the Friedrichs extension and the
maximal choice $K_2=J^*$ corresponds to the Kre\u{\i}n extension of
$T^*(I+iB)T$. Therefore the above procedure in this sense covers the
extreme maximal sectorial extensions of $T^*(I+iB)T$.

\begin{theorem}\label{s-F-N}
Let $S=T^*(I+iB)T$, $B_m$, $Q$, and $J$ be as in
Lemma~\ref{lemma4a}. Then the following statements hold.

\begin{enumerate}\def\labelenumi{\rm (\roman{enumi})}
\item
The Friedrichs extension $S_{F}$ of $S$ is given by
\[
S_{F} =Q^{\ast} (I + i B_m) Q^{\ast \ast}
\]
and the corresponding closed form $\st_{F}$ is given by
\[
\st_{S_F} [ h , k ] = \left(
                       (I + i B_m) Q^{\ast \ast} h ,
               Q^{\ast \ast} k
                          \right) , \quad h, \, k \in \dom Q^{\ast \ast} .
\]
\item The Kre\u{\i}n extension $S_{K}$ of $S$ is given by
\[
S_{K} = J^{\ast \ast} (I + i B_m) J^{\ast}
\]
and the corresponding closed form $\st_{S_K} $ is given by
\[
\st_{S_K} [ h , k ] = \left(
                        (I + i B_m) J^{\ast} h , J^{\ast} k
                          \right), \quad h, \, k \in \dom J^{\ast}.
\]
\end{enumerate}
In particular, $S_K$ is an operator if and only if $T$ is densely
defined. Therefore, $S=T^*(I+iB)T$ admits a maximal sectorial
operator extension, precisely when $T$ is densely defined; here $T$
need not be a closable operator, and it can even be multivalued.
\end{theorem}

\begin{proof}
(i) According to Proposition \ref{Msecforms} $H=Q^{\ast} (I + i B_m)
Q^{\ast \ast}$ is a maximal sectorial extension of $S$. In order to
show that it coincides with $S_F$ it suffices to prove that $\dom H\subset \dom \st_{S_F}$;
see e.g. \cite[Theorem~7.3]{HSSW17}.
Let $h  \in \dom H$. Then $\{h, h' \} \in Q^{\ast} (I + i B_m) Q^{\ast\ast}$ for some $h'\in \sH$.
In particular, $h \in \dom Q^{\ast \ast}$ and $\{h, Q^{**}h\}$
can be approximated by a sequence of elements
\[
 \left\{ \varphi_{n} , \alpha_n \right\} \in Q,
\]
where $\alpha_n\in \sM_0$ and $\{(I+iB_m)\alpha_n,\varphi_n'\}\in J\subset Q^*$ such that
\begin{equation}\label{s-vier}
 \varphi_n \to h \mbox{ in } \sH, \quad \alpha_n \to Q^{**}h \mbox{ in } \sM;
\end{equation}
see Lemma \ref{lemma4a} and \eqref{M0Def}.
Hence $(\alpha_n)$ is a Cauchy sequence in $\sM$ and this yields
\begin{equation}\label{Cauc}
 \left((\varphi_{n}'-\varphi_{m}', \varphi_{n}-\varphi_{m}\right)
 =\left((I+iB_m) (\alpha_{n}-\alpha_{m}), \alpha_{n}-\alpha_{m} \right)\to 0, \quad n,m\to \infty.
\end{equation}
Since $\{\varphi_n,\varphi_n'\} \in  J(I+iB_m)Q=S$ by Lemma \ref{lemma4a}, it follows from
\eqref{s-vier} and \eqref{Cauc} by the definition of the form $\st_{S_F}$ that $h \in \dom \st_{S_F}$;
cf. e.g. \cite[Eq. (7.2)]{HSSW17}.
Hence $\dom H \subset \dom \st_{S_F}$ and the claim $H=S_F$ is proved.

(ii) Likewise $H=J^{\ast \ast} (I + i B_m) J^{\ast}$ is a maximal sectorial extension of $S$ by Proposition \ref{Msecforms}.
To show that $H=S_K$, it suffices to prove that $\ran H\subset \dom \st_{(S^{-1})_F}$;
see \cite[Theorem~7.5]{HSSW17}.
Let $h' \in \ran H$. Then $\{h, h' \} \in J^{\ast \ast} (I + i B_m) J^{\ast}$ for some $h\in \sH$, and
\[
\{ (I +i B_m) J^{\ast} h , h' \} \in J^{\ast \ast}.
\]
This element can be approximated by a sequence of elements
\[
\left\{ (I + i B_m)\alpha_n , \varphi_{n}' \right\} \in J,
\]
where $\alpha_n\in \sM_0$ and $\{\varphi_n,\alpha_n\}\in Q\subset J^*$ for some $\varphi_n\in \dom T$, such that
\begin{equation}\label{s-een00}
 \varphi_n' \to h' \mbox{ in } \sH,
 \quad (I + i B_m)\alpha_n \to (I + i B_m) J^{\ast} h \mbox{ in } \sM;
\end{equation}
see \eqref{M0Def} and Lemma \ref{lemma4a}.
Since $B_m$ is bounded and selfadjoint in $\sM$,
the operator $I + i B_m$ is bounded with bounded inverse and,
therefore, \eqref{s-een00} is equivalent to
\begin{equation}\label{s-een}
 \varphi_n' \to h' \mbox{ in } \sH,
 \quad \alpha_n =J^* \varphi_n \to J^*h \mbox{ in } \sM.
\end{equation}
In particular, $(\alpha_n)$ is a Cauchy sequence in $\sM$ and again \eqref{Cauc} follows.
Since $\{\varphi_n,\varphi_n'\} \in J(I+iB_m)Q = S$ (see Lemma \ref{lemma4a}), it follows from
\eqref{Cauc} and \eqref{s-een} that $h' \in \dom \st_{(S^{-1})_F}$. Therefore, $\ran
H \subset \dom \st_{(S^{-1})_F}$ and $H=S_{K}$ is proved.

The last statement follows from the minimality of $S_K$, which
implies in particular that $\dom \st_H\subset \dom \st_{S_K}$: if $H$ is any
maximal sectorial operator extension of $S$, then $H$ and,
therefore, also $S_K$ is densely defined; notice that $\mul S_K=\mul J^{**}=\mul T^*$.
\end{proof}

The maximal sectorial extensions  $K^*(I+iB_m)K^{**}$ of the
sectorial relation $S$ as described in Proposition \ref{Msecforms}
with $B_m$ as in \eqref{s-B} and $Q \subset K \subset J^*$ can be
characterized among all maximal sectorial extensions of $S$.
The main ingredient in Proposition \ref{Msecforms} is that the maximal sectorial
extensions of $S$ of the form $T^{\ast} (I+iB_m) T$ with $B_m$ as in
\eqref{s-B} and $T$ an arbitrary closed linear operator satisfying
$Q\subset T \subset J^*$ can be identified as the class of all
\emph{extremal sectorial extensions of $S$}; for details see
\cite[Theorems~8.4,~8.5]{HSSW17}.

This subsection is finished with an example illustrating some special choices for $T$ with descriptions of the mappings $Q$ and $J$
appearing in the description of the maximal sectorial extensions $S_F$ and $S_K$ of the sectorial relation $S=T^*(I+iB)T$.

\begin{example}
(a) Let $T$ be an operator and consider the form
\[
 \st[h,k]=((I+iB)Th,Tk), \quad h,k\in \dom T.
\]
Then this form is $T$ is closable (closed) if and only if $T$ is
closable (closed, respectively), in which case the closures are
related by
\[
 \wt{\st}[h,k]=((I+iB)T^{**}h,T^{**}k), \quad h,k\in \dom T^{**},
\]
and one has the equalities $Q^{**}=J^*=T^{**}$ and, consequently,
\[
 S_F=S_K=T^{*}(I+iB)T^{**},
\]
which is an operator if and only if $T$ is densely defined.

(b) Let $T$ be a singular operator (or singular relation); for definitions see e.g. \cite{HSeSn2018}.
Then $\cran T=\mul T^{**}$ and $\cdom T=\ker T^{**}$. In this case
$\sM=\{0\}$ and hence,
\[
 Q=0\uphar\dom (T^*(I+iB)T)=0\uphar\ker T,\qquad \dom
 J=\{0\},\quad \mul J=\mul T^*,
\]
so $\dom Q=\ker Q$ while $J$ is a pure relation. Consequently,
\[
 S_F=Q^* Q^{**} ,\quad  S_K=J^{**} J^*
\]
are nonnegative selfadjoint relations with $\dom S_F=\ker
S_F=\overline{\ker T}$, $\dom S_K=\ker S_K=\cdom T$. If, in
addition, $T$ is densely defined, then $S_K=0$ is a selfadjoint
operator, while $S_F$ is an operator if and only if $\ker T$ is
dense in $\sH$.

(c) Let $T$ be a densely defined (not necessarily closable) operator
or relation. Then $J:\sM\to \sH$ is densely defined and since $\mul
J=\mul J^{**}=\mul T^*=\{0\}$, the Krein extension $S_K$ is a
densely defined maximal sectorial operator:
\[
 S_K= J^{\ast \ast} (I + i B_m) J^{\ast};
\]
cf. Theorem~\ref{s-F-N}.
\end{example}

\subsection{Connection to the abstract construction}
In this section the explicit construction of maximal sectorial extensions for
$S=T^*(I+iB)T$ that was using the factorized form of $S$ is connected
with the construction appearing in the abstract setting where the
specific form of $S$ is taken into account.

The starting point here follows the construction presented in
\cite{HSSW17}. With any sectorial relation $S$ in $\sH$ introduce
the range space $\ran S$ in $\sK$ and provide it with a new inner
product. Let $\{\varphi, \varphi'\}, \{\psi, \psi'\} \in S$ and
define
\begin{equation}\label{SinnerPr}
 \langle \varphi', \psi' \rangle_{S}= \frac{1}{2} \left( (\varphi',\psi)+(\varphi, \psi')
 \right).
\end{equation}
Note that if $\{\varphi_{0}, \varphi'\}, \{\psi_{0}, \psi'\}  \in S$
the inner product remains the same.
Due to the definition of $\{\varphi, \varphi'\}, \{\psi, \psi'\}
\in S$ one sees that
\[
\langle \varphi', \varphi' \rangle_{S}= \RE (\varphi',\varphi).
\]
Now sectoriality of $S$ combined with an application of the
Cauchy-Schwarz inequality (see \cite{HSSW17} for details) shows that
the isotropic part of $\ran S$ with respect to the inner product $\langle
\cdot, \cdot \rangle_S$ is given by
\[
 \sR_{0}=\{ \varphi' \in \ran S: (\varphi',\varphi)=0
  \mbox{ for some } \varphi \mbox{ with }  \{\varphi, \varphi'\} \in S\},
\]
in particular, $\sR_0=\ran S \cap \mul S^*$. Let $(\sH_{S},\langle
\cdot, \cdot\rangle_{S})$ be the Hilbert space completion of $\ran
S/ \sR_{0}$ with respect to the inner product generated on the factor space
by \eqref{SinnerPr}. Define the symmetric form $\sb$ on $\dom \sb =
\ran S/\sR_0$ by
\[
\sb [[\varphi'],[\psi']] = \frac{i}{2}
                   \left( (\varphi,\psi') - (\varphi' , \psi) \right),\quad
\{\varphi,\varphi'\}, \,\{\psi,\psi'\} \in S.
\]
Note that this definition is correct as seen by checking it for
$\{\varphi_{0}, \varphi'\}, \{\psi_{0}, \psi'\}  \in S$. It follows
from \cite{HSSW17}
 that $\sb$ is a bounded everywhere defined symmetric form
on $\ran S / \sR_{0}$.
Therefore its closure, also denoted by $\sb$, is an
everywhere defined bounded symmetric
form on $\sH_{S}$. Hence there exists a bounded selfadjoint operator
$B_S \in \mathbf{B}(\sH_{S})$ such that
\begin{equation}\label{s-B}
\sb [[\varphi'],[\psi']] =\langle B_{S} [\varphi'], [\psi'] \rangle_S,
\quad \{\varphi,\varphi'\}, \,\{\psi,\psi'\} \in S.
 \end{equation}

Now the prescribed form $T^*(I+iB)T$ of $S$ will be incorporated in
the above abstract construction. For this purpose recall that for each $\varphi', \psi' \in \ran S$ there
exists unique elements $\alpha, \beta \in \sK$ with
\begin{equation}\label{4eqB}
 \{\varphi, \alpha\} \in T, \quad \{ (I+iB) \alpha , \varphi'\} \in T^{*},
 \quad
 \{\psi, \beta\} \in T, \quad \{ (I+iB) \beta , \psi'\} \in T^{*}.
\end{equation}
see \eqref{4eq2}. This leads to
\[
 \langle \varphi', \psi' \rangle_{S}=  (\alpha, \beta),
\]
showing again that the definition is independent of the particular
first entries in $\{\varphi, \varphi'\}, \{\psi, \psi'\} \in S$.
Furthermore, \eqref{4eqB} implies that
\[
 (\varphi',\varphi)=0 \quad \Leftrightarrow \quad (\alpha, \alpha)+i(B\alpha, \alpha)=0
 \quad \Leftrightarrow \quad \alpha=0.
\]
Thus $\sR_{0}=\mul T^{*}=\mul S$ and on $\ran S / \sR_{0}$ one has
\begin{equation}\label{iso}
  \langle [\varphi'], [\psi'] \rangle_{S}=  (\alpha, \beta), \quad
  \langle [\varphi'], [\varphi'] \rangle_{S}=  (\alpha, \alpha).
\end{equation}
Furthermore, it follows from \eqref{s-B} that the bounded symmetric form
$\sb$ defined on $\dom \sb = \ran S/\sR_0$ satisfies
\[
 \sb[  [\varphi'], [\psi'] ]=(B \alpha, \beta), \quad \{\varphi, \varphi'\}, \{\psi, \psi'\} \in
 S.
\]
In other words,
\begin{equation}\label{beh}
  \langle B_{S} [\varphi'], [\psi'] \rangle_{S}=(B \alpha, \beta),
 \quad \{\varphi, \varphi'\}, \{\psi, \psi'\} \in S.
\end{equation}
Now consider the linear space $\sM_0 \subset \sK$ defined in \eqref{M0Def},
\[
 \sM_0=\{ \alpha \in \sK: \, \alpha \in \ran T, (I+iB)\alpha \in \dom T^* \},
\]
equipped with the original topology of $\sK$. Moreover, define the
mapping $\imath_0$ from $\sM_0$ onto $\ran S/\sR_0$ by
\[
 \imath_0 \alpha =[\varphi'].
\]
It follows from \eqref{iso} that $\imath_0$ is an isometry. Hence
the closure $\imath$ is a closed isometric operator from the Hilbert
space $\sM$, the closure of $\sM_0$, onto the Hilbert space $\sH_S$.
Moreover, \eqref{beh} shows that
\[
 B_m:=P_\sM B\uphar \sM=\imath^* B_S\, \imath \in \mathbf{B}(\sM).
\]
This gives the connection between the space $\sH_S$ and the operator
$B_S$ appearing in the abstract construction in \cite{HSSW17} and
the compression $B_m$ of the prescribed operator $B$ to the subspace
$\sM$.

\begin{remark}
The relations $\wt Q=\imath Q$ from $\sH$ to $\sH_S$ and $\wt J=J
\imath^*$ from $\sH_S$ to $\sH$ are the abstract counterparts of $Q$
and $J$ occurring in \cite{HSSW17} when constructing maximal
sectorial extensions for a sectorial relation $S$.
\end{remark}

\section{\textbf{Form sums of maximal sectorial relations}}\label{sec3}

As indicated in Section \ref{sec1} the treatment of the sum of two closed sectorial forms gives rise to the notion
of form sum extension of the sum of the representing maximal sectorial relations $H_1$ and $H_2$.
In order to the study the form sum extension more closely one needs to study the class of
all maximal sectorial extensions of the sum $H_1+H_2$.

Let $H_{1}$ and $H_{2}$ be maximal sectorial relations in a Hilbert space $\sH$. Then the sum $H_{1}+H_{2}$
is a sectorial relation  in $\sH$ with
\[
\dom (H_{1}+H_{2})=\dom H_{1} \cap \dom H_{2},
\]
so that the sum is not necessarily densely defined. In particular,
$H_1+H_2$ and its closure need not be operators. In fact, one sees that
\begin{equation}
\label{s-sum-mulll}
 \mul (H_{1}+H_{2})=\mul H_{1} + \mul H_{2}.
\end{equation}
To describe the class of maximal sectorial extension of $H_1+H_2$ some basic notations are fixed in Section \ref{sec3.1}.
The Friedrichs extension and Kre\u{\i}n extension of $H_1+H_2$  and, more generally, all extremal maximal sectorial extensions of $H_1+H_2$
and their factorizations are then described in Section \ref{sec3.2} and finally in Section \ref{sec3.3} the form sum extension of $H_1+H_2$
and its relation to the extremal maximal sectorial extensions of $H_1+H_2$ will be investigated.

\subsection{Pairs of maximal sectorial relations}\label{sec3.1}

According to \eqref{s-second} the maximal sectorial relations
$H_{1}$ and $H_{2}$ are decomposed as follows
\begin{equation}\label{H12}
 H_{j} = A_{j}^{\half} (I+iB_{j}) A_{j}^{\half},
\quad 1 \le j \le 2,
\end{equation}
where $A_{j}$ (the real part of $H_{j}$), $1 \le j \le 2,$ are
nonnegative selfadjoint relations in $\sH$ and $B_{j}$, $1 \le j \le
2,$ are (unique) bounded selfadjoint operators in $\sH$; see \eqref{Bunique} in Lemma \ref{s-second}.
Furthermore, if $A_{1}$ and $A_{2}$ are decomposed as
\[
  A_{j}=A_{js} \oplus A_{j \infty}, \quad 1 \le j \le 2,
\]
where $A_{j \infty} =\{0\} \times \mul A_{j}$,
$1 \le j \le 2$, $A_{js}$, $1 \le j \le 2$ are densely
defined nonnegative selfadjoint operators
(defined as orthogonal complements in the graph sense),
then the uniquely determined square roots of $A_{j}$, $1 \le j \le 2$
are given by
\[
   A_{j}^\half=A_{js}^\half \oplus A_{j \infty},
   \quad 1 \le j \le 2.
\]
Associated with $H_{1}$ and $H_{2}$ is the relation
$\Phi$ from $\sH \times \sH$ to $\sH $, defined by
\begin{equation}
\label{s-s-Einz}
 \Phi= \left\{\, \left\{ \{f_{1},f_{2}\}, f_{1}' + f_{2}' \right\}
 :\, \{f_{j},f_{j}' \} \in A_{j}^{\half} , \,
 1 \le j \le 2\,\right\}.
\end{equation}
Clearly, $\Phi$ is a relation whose domain and multivalued part
are given by
\[
 \dom \Phi = \dom A_{1}^{\half} \times \dom A_{2}^{\half},
 \quad \mul \Phi=\mul H_{1} +\mul H_{2}.
\]
The relation $\Phi$ is not necessarily densely defined in $\sH
\times \sH$, so that in general $\Phi^*$ is a relation  as $\mul
\Phi^*=(\dom \Phi)^\perp$.
Furthermore, the adjoint $\Phi^*$ of $\Phi$
is the relation from $\sH$ to $\sH \times \sH$, given by
\begin{equation}
\label{s-s-Zwei}
\Phi^* = \left\{ \left\{ h, \{ h_{1}', h_{2}'\} \right\} \, :
\{h, h_{j}'\} \in A_{j}^{\half}, \, 1 \le j \le 2 \right\}.
\end{equation}
The identity \eqref{s-s-Zwei} shows that the (orthogonal)
operator part
$(\Phi^*)_s$ of $\Phi^*$ is given by:
\begin{eqnarray}
\label{s-s-qus}
 (\Phi^*)_s & = &\left\{ \left\{ h, \{ h_{1}', h_{2}'\} \right\} \,
 : \{h, h_{j}'\} \in A_{js}^{\half}, \, 1 \le j \le 2  \right\}
 \\
& = & \left\{ \left\{ h, \{ A_{1s}^{\half} h, A_{2s}^{\half} h\} \right\}
\, : h \in \dom A_{1}^{\half} \cap \dom A_{2}^{\half} \right\}.
\nonumber
\end{eqnarray}
The identities \eqref{s-s-Zwei} and \eqref{s-s-qus} show that
\[
\dom  \Phi^{*}  = \dom A_{1}^{\half} \cap \dom A_{2}^{\half},
\, \, \ran (\Phi^{*})_s  = \sF_{0},
\, \, \mul \Phi^*=\mul H_{1} \times \mul H_{2},
\]
where the subspace $\sF_{0} \subset \sH \times \sH$ is defined by
\begin{equation}\label{F0}
\sF_{0} = \left\{ \left\{ A_{1s}^{\half} h , A_{2s}^{\half} h
\right\} \, : \, h \in \dom A_{1}^{\half} \cap \dom A_{2}^{\half}
\right\}.
\end{equation}
The closure of $\sF_{0}$ in $\sH \times \sH$ will be denoted by
$\sF$. Define the relation $\Psi$ from $\sH$ to $\sH \times \sH$
by
\begin{equation}
\label{s-s-qu}
 \Psi=\left\{\, \left\{h, \left\{A_{1s}^{\half} h, A_{2s}^{\half}
 h\right\}\right\} :\, h \in \dom H_{1} \cap \dom H_{2} \,\right\} \subset
 \sH \times (\sH \times \sH).
\end{equation}
It follows from this definition that
\[
 \dom \Psi=\dom H_{1} \cap \dom H_{2}, \quad \ran \Psi=\sE_0, \quad
  \mul \Psi=\{0\},
\]
where the space $\sE_0 \subset \sH \times \sH$ is defined by
\begin{equation}\label{E0}
 \sE_0  =  \left\{\, \left\{A_{1s}^{\half}f , A_{2s}^{\half}f\right\}:\,
        f\in \dom H_{1} \cap \dom H_{2} \,\right\}.
\end{equation}
Observe that $\sE_0 \subset \sF_0$.
The closure of $\sE_{0}$ in $\sH \times \sH$ will be denoted by
$\sE$. Hence,
\begin{equation}
\label{s-s-ef}
          \sE \subset \sF.
\end{equation}
Comparison of \eqref{s-s-qus} and \eqref{s-s-qu} shows
\[
\Psi \subset (\Phi^*)_s,
\]
and thus the operator $\Psi$ is closable. It follows from
$\cdom \Psi^*=(\mul \Psi^{**})^\perp$ and
$\mul \Psi^*=(\dom \Psi)^\perp$ that
\[
  \cdom \Psi^*=\sH , \quad \mul \Psi^* =(\dom H_{1} \cap \dom H_{2})^\perp.
\]
Next define the relation $K $ from $\sH \times \sH$ to $\sH$ by
\begin{eqnarray}
\label{s-s-opK}
K & = &
\{ \{ \{
(I+iB_{1})A_{1s}^{\half} f, (I+iB_{2})A_{2s}^{\half}f \}, f_{1}' + f_{2}' \}\, :
 \\
&&
\{ (I+iB_{1})A_{1s}^{\half} f, f_{1}' \} \in A_{1}^{\half},
\{ (I+iB_{2})A_{2s}^{\half} f, f_{2}' \} \in A_{2}^{\half} \}
\nonumber \\
&&
\subset (\sH \times \sH) \times \sH. \nonumber
\end{eqnarray}
Clearly, the domain and multivalued part of $K$ are
given by
\[
 \dom K= \sD_0, \quad \mul K=\mul (H_{1} + H_{2}),
\]
where
\begin{equation}\label{domK}
   \sD_{0}= \left\{\,\{(I+iB_{1}) A_{1s}^{1/2}f,
   (I+iB_{2}) A_{2s}^{1/2}f\} :
   \,  f\in \dom H_{1} \cap \dom H_{2}
  \,\right\} .
\end{equation}
The closure of $\sD_{0}$ in $\sH \times \sH$ will be denoted by
$\sD$.

\begin{lemma}\label{lem3.1}
The relations $K$, $\Phi$, and $\Psi$ satisfy the following
inclusions:
\begin{equation}
\label{s-s-trits}
  K \subset \Phi \subset \Psi^*, \quad \Psi \subset \Phi^* \subset
K^*.
\end{equation}
\end{lemma}

\begin{proof}
To see this note that $K \subset \Phi$ follows from \eqref{s-s-Einz}
and \eqref{s-s-opK}, and that $\Psi \subset \Phi^*$ follows from
\eqref{s-s-Zwei} and \eqref{s-s-qu}. Therefore, also
$\Phi^* \subset K^*$ and $\Phi \subset \Phi^{**} \subset \Psi^*$.
\end{proof}

\subsection{The Friedrichs and the Kre\u{\i}n extensions of $H_1+H_2$}\label{sec3.2}

Let $H_{1}$ and $H_{2}$ be maximal sectorial relations in a Hilbert
space $\sH$. Since (the closure of) the sectorial sum
$H_{1}+H_{2}$ has equal defect numbers, (the closure of) the sum
$H_{1}+H_{2}$ has maximal sectorial extensions in $\sH$.
Two of them, the Friedrichs extension and the Kre\u{\i}n extension $(H_{1}+H_{2})_F$ and $(H_{1}+H_{2})_K$
as maximal sectorial relations have factorizations as $H_1$ and $H_2$ in \eqref{H12}.
A natural problem is to express such factorizations in terms of the initial relations $H_{1}$ and $H_{2}$.

Introduce the orthogonal sum of the operators $B_{1}$ and $B_{2}$ in $\sH \times \sH$ by
\[
 B_\oplus:=B_1 \oplus B_2 =
\begin{pmatrix}
B_{1} & 0 \\
0 & B_{2}
\end{pmatrix}.
\]
This shorthand notation is used to shorten some of the forthcoming formulas.
The aim in the description of $(H_{1}+H_{2})_F$ and $(H_{1}+H_{2})_K$ is to keep the presentation
as explicit as possible by incorporating the initial data on the factorizations \eqref{H12} of $H_1$ and $H_2$
directly via the mappings $\Phi$, $\Psi$, and $K$ in Subsection \ref{s-sum-mulll}.

Now proceed to the construction of the Friedrichs extension for the sum $H_1+H_2$.

\begin{theorem}\label{ss-twee}
Let $H_{1}$ and $H_{2}$ be maximal sectorial and let $\Psi$ be defined by \eqref{s-s-qu}.
The Friedrichs extension of $H_{1}+H_{2}$ is given by
\[
 (H_{1}+H_{2})_F=\Psi^* (I + i B_\oplus) \Psi^{**}
\]
and the corresponding form is given by
\[
\st_{F} [f,g]= ((I + iB_\oplus ) \Psi^{**} f, \Psi^{**} g),
\quad f, \, g  \in \dom \st_{F} = \dom \Psi^{**}.
\]
\end{theorem}

\begin{proof}
First it is shown that the relation
$\Psi^* (I + i B_\oplus ) \Psi^{**}$ extends the relation
$S:=H_{1} + H_{2}$. Let
$\{ h, h_{1}'+h_{2}' \} \in  H_{1} + H_{2}$
for some $\{ h, h_{1}' \} \in H_{1}$ and
$\{ h, h_{2}' \} \in H_{2}$.
Thus,
\[
\{ h, \{A_{1s}^{1/2} h, A_{2s}^{1/2} h \}
\in \Psi \subset \Psi^{**},
\]
and also
\[
\{ \{ (I + iB_{1}) A_{1}^{1/2} h, (I + iB_{2}) A_{2}^{1/2} h\},
h_{1}'+h_{2}'  \} \in K \subset \Psi^{*},
\]
as can be verified directly
\[
\langle
\{
\{(I + iB_{1})A_{1s}^{1/2} h, (I + iB_{2})A_{2}^{1/2} h \},
h_{1}'+h_{2}' \},
\{ \varphi, \{A_{1s}^{1/2} \varphi, A_{2s}^{1/2} \varphi \} \}
\rangle =0,
\]
for all $\varphi \in \dom S=\dom H_{1} \cap \dom H_{2}$.
Therefore
$S \subset \Psi^{*} (I + i B_\oplus ) \Psi^{**}$.

Now let $\{ f, g\} \in \Psi^{*} (I + i B_\oplus ) \Psi^{**}$,
so that $\{f, h \} \in \Psi^{**}$ and
$\{ (I +i B_\oplus )h, g\} \in \Psi^{*}$ for some
$h \in \sH \times \sH$. Since $\Psi^{**}$ is the closure
of $\Psi$ there exists a sequence of elements
$f_{n} \in \dom \Psi=\dom S$ such that
\begin{equation}
\label{ss-mde*}
f_{n} \to f, \quad \Psi f_{n} \to h,
\quad
\text{as } \, n \to \infty.
\end{equation}
It follows from $\{ f,h\} \in \Psi^{**}$ and
$\{ (I + i B_\oplus )h, g\} \in \Psi^{*}$ that
\[
(g,f) = (h,h) + i (B_\oplus  h,h),
\]
which implies that
\begin{equation}
\label{ss-ca3}
\RE (g,f) = (h,h).
\end{equation}
Similarly it follows from $\{ f,h\} \in \Psi^{**}$
and
\[
\{ \{ (I + iB_{1}) A_{1s}^{1/2} f_{n},
(I + iB_{2}) A_{2s}^{1/2} f_{n} \} ,
(H_{1s}+H_{2s}) f_{n}
\} \in \Psi^{*}
\]
that
\begin{equation}
\label{ss-ajuta3}
((H_{1s} + H_{2s}) f_{n}, f)
=
( \{ A_{1s}^{1/2} f_{n}, A_{2s}^{1/2} f_{n} \},h )
+ i ( \{ B_{1} A_{1s}^{1/2} f_{n}, B_{2} A_{2s}^{1/2} f_{n} \},h ).
\end{equation}
Likewise, it follows from
$\{ f_{n}, \{ A_{1s}^{1/2} f_{n},
A_{2s}^{1/2} f_{n} \} \} \in \Psi \subset \Psi^{**}$
and
$\{ (I + i B_\oplus ) h, g\} \in \Psi^{*}$ that
\begin{equation}
\label{ss-ajuta4}
(g, f_{n})
=
( h, \{ A_{1s}^{1/2} f_{n}, A_{2s}^{1/2} f_{n} \})
+ i ( B_\oplus  h, \{ A_{1s}^{1/2} f_{n}, A_{2s}^{1/2} f_{n} \}).
\end{equation}
A combination of \eqref{ss-ajuta3} and
\eqref{ss-ajuta4} leads to
\begin{eqnarray}
\label{ss-ca4}
\RE (
((H_{1s} + H_{2s}) f_{n}, f) +
(g, f_{n})
)
& = &
\RE (
( h, \{ A_{1s}^{1/2} f_{n}, A_{2s}^{1/2} f_{n} \})
\nonumber \\
& &
+
( \{ A_{1s}^{1/2} f_{n}, A_{2s}^{1/2} f_{n} \},h )
).
\end{eqnarray}
This leads to the following identity
\begin{eqnarray}
  \|h- \Psi f_{n} \|^{2}
 &= & \|h \|^{2}
   - ( h, \{ A_{1s}^{1/2} f_{n}, A_{2s}^{1/2} f_{n} \})
   \nonumber \\
 &   & -
   (\{ A_{1s}^{1/2} f_{n}, A_{2s}^{1/2} f_{n} \} , h)
   + \|\{ A_{1s}^{1/2} f_{n}, A_{2s}^{1/2} f_{n} \} \|^{2}
   \nonumber \\
 & = & \|h \|^{2}
   - \RE ( h, \{ A_{1s}^{1/2} f_{n}, A_{2s}^{1/2} f_{n} \})
   \nonumber \\
 &   & -
   \RE (\{ A_{1s}^{1/2} f_{n}, A_{2s}^{1/2} f_{n} \} , h)
   + \|\{ A_{1s}^{1/2} f_{n}, A_{2s}^{1/2} f_{n} \} \|^{2}
   \nonumber \\
 & = & \RE (g - (H_{1s} + H_{2s}) f_{n}, f-f_n),
 \nonumber
\end{eqnarray}
where \eqref{ss-ca3}, and \eqref{ss-ca4} have been used,
respectively. Therefore \eqref{ss-mde*} implies that
\begin{equation}
\label{ss-mde**}
  f_{n} \to f ,
  \quad
  \RE (g - (H_{1s} + H_{2s}) f_{n}, f-f_n) \to 0.
\end{equation}
Since $f_n \in \dom S$, it follows from \eqref{ss-mde**} and the definition of $S_F$ that
$\{f,g\} \in S_{F}$.
Hence, $\Psi^{*}(I + iB)\Psi^{**} \subset S_{F}$, and
since $\Psi^{*}(I + iB)\Psi^{**}$ and $S_{F}$ are both
maximal sectorial, the identity
$\Psi^{*}(I + iB)\Psi^{**}=S_{F}$ follows.
The statement concerning the associated closed form $\st_{F}$ follows from
the first representation theorem and the definition of $S_F$; cf. \cite[Theorem 5.1]{HSSW17}.
\end{proof}

Next the construction of the Kre\u{\i}n extension for the sum $H_1+H_2$ is given.

\begin{theorem}\label{KVNext}
Let $H_{1}$ and $H_{2}$ be maximal sectorial and let $K$ be defined by \eqref{s-s-opK}.
The Kre\u{\i}n extension of $H_{1} + H_{2}$ is given by
\[
 (H_{1} + H_{2})_K=K^{**}(I + i B_\oplus)K^{*}.
\]
If, in addition, $\sE=\clos\sE_0$ and $\sD=\clos\sD_0$ (see \eqref{E0}, \eqref{domK})
satisfy the equality $\sE=\sD$ then the corresponding closed sectorial form is given by
\[
\st_{K} [f,g]= ((I + iB_\oplus ) (K^{*})_{s} f, (K^{*})_{s} g),
\quad
f, \, g  \in \dom \st_{K} = \dom K^{*}.
\]
\end{theorem}

\begin{proof}
Assume that $\{ f, f_{1}' +f_{2}'\} \in  S=H_{1} + H_{2}$,
with $\{ f, f_{1}'\} \in  H_{1}$ and $\{ f, f_{2}'\} \in  H_{2}$.
This implies that
\[
\{ \{ (I + iB_{1})A_{1s}^{1/2} f, (I + iB_{2})A_{2s}^{1/2} f\},
f_{1}' +f_{2}'\} \in K \subset K^{**}.
\]
Moreover,
\[
\{ f, \{A_{1s}^{1/2} f, A_{2s}^{1/2} f\}\} \in K^{*},
\]
as can be verified directly
\[
\langle
\{ f, \{A_{1s}^{1/2} f, A_{2s}^{1/2} f \} \},
\{ \{(I + iB_{1}) A_{1s}^{1/2} \varphi,
   (I + iB_{2})  A_{2s}^{1/2} \varphi \},
   f_{1}' +f_{2}' \}
\rangle =0 ,
\]
for all $\varphi \in \dom S=\dom H_{1} \cap \dom H_{2}$.
Therefore $S \subset K^{**} (I + i B_\oplus ) K^{*}$.

Now assume that $\{ f, g\} \in K^{**} (I + i B_\oplus) K^{*}$.
This means that $\{f, h \} \in K^{*}$ and
$\{ (I +i B_\oplus )h, g\} \in K^{**}$ for some
$h \in \sH \times \sH$. Since $K^{**}$ is the closure
of $K$ there exists a sequence of elements
$\{ \varphi_{n} , \varphi_{n}'\} \in K$ with
\[
 \{ \varphi_{n}, \varphi_{n}' \} \rightarrow
 \{ (I + i B_\oplus ) h, g\} \in K^{**}, \quad \text{as } \, n \to \infty.
\]
Clearly,
\[
\varphi_{n}
=
\{ (I + i B_{1}) A_{1s}^{1/2} f_{n},
   (I + i B_{2}) A_{2s}^{1/2} f_{n} \},
   \quad \varphi_{n}' = f_{n1}' + f_{n2}'
\]
for some $\{ f_{n} , f_{n1}' \} \in H_{1}$
and $\{ f_{n} , f_{n2}' \} \in H_{2}$.
Therefore,
\begin{equation}
\label{ss-mda*}
\{ A_{1s}^{1/2} f_{n}, A_{2s}^{1/2} f_{n} \}
\rightarrow h, \quad
f_{n1}' + f_{n2}' \rightarrow g, \quad
\text{as} \quad n \to \infty.
\end{equation}
It follows from $\{ f,h\} \in K^{*}$ and
$\{ (I + i B_\oplus )h, g\} \in K^{**}$ that
\[
(g,f) = (h,h) + i (B_\oplus  h,h),
\]
which implies that
\begin{equation}
\label{ss-ca1}
\RE (g,f) = (h,h).
\end{equation}
On the other hand, $\{ f,h\} \in K^{*}$
and $\{ \varphi_{n}, \varphi_{n}' \} \in K^{**}$ leads to
\begin{equation}
\label{ss-ajuta1}
(f_{n1}' +f_{n2}' , f)
=
( \{ A_{1s}^{1/2} f_{n}, A_{2s}^{1/2} f_{n} \},h )
+ i ( B_\oplus  \{ A_{1s}^{1/2} f_{n}, A_{2s}^{1/2} f_{n} \},h ).
\end{equation}
Similarly it follows from
$\{ f_{n}, \{ A_{1s}^{1/2} f_{n},
A_{2s}^{1/2} f_{n} \} \} \in K^{*}$
and
$\{ (I + i B_\oplus ) h, g\} \in K^{**}$ that
\begin{equation}
\label{ss-ajuta2}
(g, f_{n})
=
( h, \{ A_{1s}^{1/2} f_{n}, A_{2s}^{1/2} f_{n} \})
+ i ( B_\oplus  h, \{ A_{1s}^{1/2} f_{n}, A_{2s}^{1/2} f_{n} \}).
\end{equation}
Now a combination of \eqref{ss-ajuta1} and
\eqref{ss-ajuta2} shows that
\begin{eqnarray}
\label{ss-ca2}
\RE (
(f_{n1}' + f_{n2}', f) +
(g, f_{n})
)
& = &
\RE (
( h, \{ A_{1s}^{1/2} f_{n}, A_{2s}^{1/2} f_{n} \})
\nonumber \\
&&
+
( \{ A_{1s}^{1/2} f_{n}, A_{2s}^{1/2} f_{n} \},h )
).
\end{eqnarray}
This leads to the following identity
\begin{eqnarray}
 \|h- \{ A_{1s}^{1/2} f_{n}, A_{2s}^{1/2} f_{n} \} \|^{2}
 & =  & \|h \|^{2}
   - ( h, \{ A_{1}^{1/2} f_{n}, A_{2}^{1/2} f_{n} \})
   \nonumber \\
 &    & - (\{ A_{1}^{1/2} f_{n}, A_{2}^{1/2} f_{n} \} , h)
   + \|\{ A_{1}^{1/2} f_{n}, A_{2}^{1/2} f_{n} \} \|^{2}
   \nonumber \\
 & = & \|h \|^{2}
   - \RE ( h, \{ A_{1}^{1/2} f_{n}, A_{2}^{1/2} f_{n} \})
   \nonumber \\
 &   & -
   \RE (\{ A_{1s}^{1/2} f_{n}, A_{2s}^{1/2} f_{n} \} , h)
   + \|\{ A_{1s}^{1/2} f_{n}, A_{2s}^{1/2} f_{n} \} \|^{2}
   \nonumber \\
 & = & \RE (g -  (f_{n1}' +f_{n2}' ), f-f_n),
   \nonumber
\end{eqnarray}
where \eqref{ss-ca1}, and \eqref{ss-ca2} have been used,
respectively. Therefore \eqref{ss-mda*} implies that
\begin{equation}
\label{ss-mda**}
  f_{n1}' + f_{n2}' \to g ,
  \quad
  \RE (g - ( f_{n1}' + f_{n2}') , f-f_n) \to 0.
\end{equation}
Since $\{ f_n,  f_{n1}' + f_{n2}' \} \in  S$,
the relation \eqref{ss-mda**} implies that
$\{f,g\} \in S_K$.
Hence, $K^{**}(I + iB_\oplus )K^* \subset S_K$, and
since $K^{**} (I + i B_\oplus )K^{*}$ and $S_K$ are both
maximal sectorial (see Proposition \ref{SectorialCor}), the identity
$K^{**} (I + i B)K^{*}=S_K$ follows.

As to the statement concerning the form $\st_{K}$ observe that
\begin{equation}\label{E0D0}
  \sD_0=\dom K=(I+i B_\oplus) \sE_0;
\end{equation}
see \eqref{E0}, \eqref{domK}. Therefore, the assumption $\sE=\sD$
implies that $\sD=\cdom K$ is invariant under the selfadjoint operator $B_\oplus$.
Then also $\mul K^*=\sD^\perp$ is invariant under $B_\oplus$ and hence it follows from
\cite[Theorem 5.1]{HSSW17} that $K^{**}(I + i B_\oplus)K^{*}=((K^{*})_{s})^*(I + i B_\oplus)(K^{*})_{s}$
and that the corresponding closed form $\st_K$ is determined by the operator part $(K^{*})_{s}$ of $K^*$.
\end{proof}

The product $K^{**} (I+iB_\oplus ) K^*$ is a maximal sectorial relation
whose multivalued part is given by $\mul K^{**}$. Therefore, it
follows from Theorem \ref{KVNext} that
\[
\mul (H_{1}+H_{2})_K=\mul K^{**} .
\]
Recall from \cite[Theorem 1]{Ar} (cf. \cite[Theorem 7.6]{HSSW17}) that the Kre\u{\i}n extension $S_K$ has the largest form domain
among all maximal sectorial extensions of a sectorial relation $S$.
In particular, this implies that the relation $S$ is ``\textit{sectorially closable}'', i.e.,
$S$ has a maximal sectorial operator extension if and only if
the Kre\u{\i}n extension $S_K$ is an operator, which in the present case holds for $S=H_1+H_2$ if and only
if the relation $K$ is a closable operator or, equivalently, $K^*$ is densely defined.

Likewise, the product $\Psi^* (I + iB_\oplus ) \Psi^{**}$ is a
maximal sectorial relation whose multivalued part is given by
$\mul \Psi^{*}=(\dom \Psi)^\perp$, so that it follows from Theorem~\ref{ss-twee} that
\[
 \mul (H_{1}+H_{2})_F=(\dom H_{1} \cap \dom H_{2})^\perp.
\]
Hence, when $H_{1}+H_{2}$ is densely defined, then $H_{1}+H_{2}$
is automatically an operator and all maximal sectorial extensions
are operators. The orthogonal operator part of
$\Psi^* (I+iB_\oplus ) \Psi^{**}$ is the maximal sectorial operator
corresponding to the closed form
\[
  ((I+iB_\oplus )\Psi^{**} h, \Psi^{**} k), \quad h,k \in \dom \Psi^{**}.
\]

The description of the closed sectorial form $\st_K$ associated with the Kre\u{\i}n extension $(H_1+H_2)_K$ in Theorem \ref{KVNext} is stated under
the additional condition $\sE=\sD$. When this condition fails to hold the description of the form $\st_K$ becomes more involved and
will be treated elsewhere; see \cite{HS2019}.
The form $\st_K$ can be used to give a complete description of all \textit{extremal maximal sectorial extensions} of the sum $H_1+H_2$.
Namely, a maximal sectorial extension $H$ of a sectorial relation $S$ is extremal precisely when the corresponding closed sectorial form
$\st_H$ is a restriction of the closed sectorial form $\st_K$ generated by the Kre\u{\i}n extension $S_K$ of $S$;
see e.g. \cite[Definition 7.7, Theorems~8.4,~8.5]{HSSW17}.
Therefore, Theorem \ref{KVNext} implies the following description of all extremal maximal sectorial extensions of $H_1+H_2$.

\begin{proposition}
\label{s-sum-caracter} Let $H_{1}$ and $H_{2}$ be
maximal sectorial relations in a Hilbert space $\sH$ and assume that $\sE=\clos\sE_0$ and $\sD=\clos\sD_0$ (see \eqref{E0}, \eqref{domK})
satisfy the equality $\sE=\sD$.
Then the following statements are equivalent:
\begin{enumerate}
\def\labelenumi{\rm (\roman{enumi})}

\item $\widetilde{H}$ is an extremal maximal sectorial extension of $H_{1}+H_{2}$;

\item $\widetilde{H} = T_{\sL}^{\ast} (I+iB_\oplus ) T_{\sL}^{\ast \ast}$, where $T_{\sL}$ is the restriction of the operator part
$(K^*)_s$ to a linear subspace $\sL$ satisfying
$$\dom (H_{1}+H_{2}) \subset \sL \subset \dom K^*.$$
\end{enumerate}
\end{proposition}

\subsection{The form sum construction}\label{sec3.3}

The maximal sectorial relations $H_{1}$ and $H_{2}$ generate the
following closed sectorial form
\begin{equation}
\label{s-s-fs}
 ((I+iB_{1})A_{1s}^{\half} h, A_{1s}^{\half} k)
 +((I+iB_{2})A_{2s}^{\half} h, A_{2s}^{\half} k),
 \quad h,k \in \dom A_{1}^{\half} \cap \dom A_{2}^{\half}.
\end{equation}
Observe that the restriction of this form to $\dom \Psi^{**}$ is
equal to
\[
 (\Psi^{**}h, \Psi^{**}k)
 =
 ((I+iB_{1})A_{1s}^{\half} h, A_{1s}^{\half} k)
 +((I+iB_{2})A_{2s}^{\half} h, A_{2s}^{\half} k),
 \quad h,k \in \dom \Psi^{**},
\]
since $\Psi^{**} \subset (\Phi^*)_s$, cf. \eqref{s-s-qus}.
Thus, the form in \eqref{s-s-fs} has a natural domain
which is in general larger than $\dom \Psi^{**}$.

\begin{theorem}
\label{s-sum-een} Let $H_{1}$ and $H_{2}$ be maximal sectorial and let $\Phi$ be defined by \eqref{s-s-Einz}.
The maximal sectorial relation
\[
 \Phi^{**} (I+iB_\oplus ) \Phi^*
\]
is an extension of the relation $H_{1}+H_{2}$, which corresponds to
the closed sectorial form in \eqref{s-s-fs}.

Assume, in addition, that $\sE=\clos\sE_0$ and $\sD=\clos\sD_0$ (see \eqref{E0}, \eqref{domK})
satisfy the equality $\sE=\sD$ and let $\sF=\clos \sF_0$ be defined by \eqref{F0}.
Then the following statements are equivalent:
\begin{enumerate}
\def\labelenumi{\rm (\roman{enumi})}

\item $\Phi^{**} (I+iB_\oplus ) \Phi^*$  is extremal;

\item $\sE = \sF$.
\end{enumerate}
\end{theorem}

\begin{proof}
By \cite[Theorem 5.1]{HSSW17} the form sum \eqref{s-s-fs} can be written as
\[
 ((I+iB_\oplus )(\Phi^{*})_s\, h,(\Phi^*)_{s}\, k), \quad h,k \in \dom
 (\Phi^{*})_{s} = \dom A_{1}^{\half} \cap \dom A_{2}^{\half},
\]
so that $\Phi^{**} (I+iB_\oplus ) \Phi^*$ is the
maximal sectorial relation in $\sH$ which corresponds to \eqref{s-s-fs}
via the first representation theorem, since $\mul \Phi^*=\mul H_1\times\mul H_2$ is clearly invariant under $B_\oplus$,
when $B_1$ and $B_2$ are the unique operators as described in Lemma \ref{s-second}.

To show that $\Phi^{**} (I+iB_\oplus ) \Phi^*$ extends
$H_{1} + H_{2}$, let $\{h, h_{1}' +
h_{2}' \} \in (H_{1}+H_{2})$ for some
$\{h, h_{1}' \} \in H_{1}$ and
$\{h, h_{2}' \} \in H_{2}$, so that
$h \in \dom H_{1} \cap \dom H_{2}$. Clearly,
$ \{h, \{A_{1s}^{\half} h, A_{2s}^{\half} h \} \} \in \Phi^*$.
Moreover,
\[
 \{ \{(I+iB_{1}) A_{1s}^{\half} h, (I+iB_{2})A_{2s}^{\half} h \},
 h_{1}' + h_{2}'\} \in  \Phi^{**},
\]
as can be verified directly:
\[
(h_{1}' + h_{2}',\varphi)-(\{(I+iB_{1})A_{1s}^{\half} h,
(I+iB_{2})A_{2s}^{\half} h \},\{A_{1s}^{\half} \varphi, A_{2s}^\half \varphi\})=0
\]
for all $\varphi \in \dom H_{1} \cap \dom H_{2}$.
Therefore
$\{h, h_{1}' + h_{2}' \} \in \Phi^{**} (I+iB_\oplus ) \Phi^*$.
Hence $H_{1} + H_{2} \subset \Phi^{**} (I+iB_\oplus ) \Phi^*$.
This proves the first statement.

Now the equivalence in the second statement will be proved.

(i) $\Rightarrow$ (ii) Since $\sE \subset \sF$ by \eqref{s-s-ef} it is enough to prove
the inclusion $\sF \subset \sE$. Assume that the form sum extension
of $H_{1} + H_{2}$ is extremal. Then by Proposition \ref{s-sum-caracter}
there exists a subspace $\sL$ such that
\begin{equation}
\label{NICE1}
 ((\Phi^{*})_{s})^{*} (I+iB_\oplus ) (\Phi^{*})_{s}
 = \Phi^{**} (I+iB_\oplus ) \Phi^{*}
 = T_{\sL}^{*} (I+iB_\oplus ) T_{\sL}^{**}.
\end{equation}
Let $P_\sE$ be the orthogonal projection of $\sH \times \sH$ onto
$\sE$. By \eqref{s-s-trits} $(\Phi^{*})_{s} \subset \Phi^{*} \subset
K^{*}$ and therefore $P_\sE(\Phi^{*})_{s}\subset P_\sE K^*=(K^*)_{s}$, since
by assumption $\sE=\sD=\cdom K$.
Moreover, $\dom P_\sE(\Phi^{*})_{s} =\dom (\Phi^{*})_{s} = \dom
T_{\sL}^{**}$ and since $P_\sE(\Phi^{*})_{s}$ and
$T_{\sL}^{**}$ are restrictions of the operator $(K^*)_{s}$ it follows
that
\[
 P_\sE(\Phi^{*})_{s} = T_{\sL}^{**},
 \quad
 ((\Phi^{*})_{s})^{*}P_\sE = T_{\sL}^{*}.
\]
The assumption $\sE=\sD$ also implies that $\sE=\cdom K$ is invariant under $B_\oplus$; see \eqref{E0D0}.
Now one obtains from \eqref{NICE1} the equalities
\begin{eqnarray}
 ((\Phi^{*})_{s})^{*} (I+iB_\oplus ) (\Phi^{*})_{s} & = & T_{\sL}^{*} (I+iB_\oplus ) T_{\sL}^{**}
 \nonumber \\
 & = & ((\Phi^{*})_{s})^{*}P_{\sE} (I+iB_\oplus ) T_{\sL}^{**}
 \nonumber \\
 & = & ((\Phi^{*})_{s})^{*} (I+iB_\oplus ) P_{\sE} T_{\sL}^{**}
 \nonumber \\
 & = & ((\Phi^{*})_{s})^{*} (I+iB_\oplus ) T_{\sL}^{**} . \nonumber
\end{eqnarray}
Hence for every $f\in\dom ((\Phi^{*})_{s})^{*} (I+iB_\oplus ) (\Phi^{*})_{s}$ one has
$$(I+iB_\oplus ) ((\Phi^{*})_{s}f-T_{\sL}^{**}f) \in \ker ((\Phi^{*})_{s})^{*}.$$
Since $(\Phi^{*})_{s}f-T_{\sL}^{**}f \in \sF=\cran
(\Phi^{*})_{s}=(\ker ((\Phi^{*})_{s})^{*})^\perp$, see \eqref{s-s-qus}, \eqref{F0}, this implies that
\[
 \left((I+iB_\oplus )((\Phi^{*})_{s}f-T_{\sL}^{**}f ), (\Phi^{*})_{s}f-T_{\sL}^{**}f \right)=0
\]
and thus
$(\Phi^{*})_{s}f-T_{\sL}^{**}f=0$. Therefore $(\Phi^{*})_{s}(\dom
((\Phi^{*})_{s})^{*} (I+iB_\oplus ) (\Phi^{*})_{s})\subset \ran T_{\sL}^{**}\subset \sE$.
Since $\dom ((\Phi^{*})_{s})^{*} (I+iB_\oplus )(\Phi^{*})_{s}$
is a core for the corresponding closed form, or equivalently, the closure of
$(\Phi^{*})_{s} \uphar\dom((\Phi^{*})_{s})^{*} (I+iB_\oplus ) (\Phi^{*})_{s}$ is equal to
$(\Phi^{*})_{s}$, the claim follows:
$\sF = \cran (\Phi^{*})_{s} \subset \sE.$

(ii) $\Rightarrow$ (i) Assume that $\sE = \sF$. Then $\sF_{0}=\ran (\Phi^{*})_{s}\subset \sE$ and
the equalities $\sE=\sD=\cdom K$ combined with $(\Phi^{*})_{s} \subset \Phi^{*} \subset
K^{*}$ imply that $(\Phi^{*})_{s}\subset P_\sE K^*=(K^*)_{s}$.
Therefore $(\Phi^{*})_{s} = T_{\sL}$ with the choice
\[
\sL = \dom (\Phi^{*})_{s} = \dom A_{1}^{\half} \cap \dom A_{2}^{\half}.
\]
Hence,
\[
 \Phi^{**} (I+iB_\oplus ) \Phi^{*}
 = ((\Phi^{*})_{s})^{*} (I+iB_\oplus ) (\Phi^{*})_{s}
 = T_{\sL}^{*} (I+iB_\oplus ) T_{\sL}^{**},
\]
which shows that $\Phi^{**} (I+iB_\oplus ) \Phi^{*}$ is  extremal,
cf. Proposition \ref{s-sum-caracter}.
\end{proof}

The maximal sectorial relation $\Phi^{**} (I+iB_\oplus ) \Phi^*$ naturally extends the factorized
sectorial relation $\Phi(I+iB_\oplus ) \Phi^*=H_1+H_2$ and, as indicated in Section \ref{sec1}
it is called the \textit{form sum extension of the sectorial
relation $H_{1} + H_{2}$} (induced by the form \eqref{s-s-fs}).
Its multivalued part is given by
$\mul \Phi^{**} (I+iB_\oplus ) \Phi^*= \mul \Phi^{**}=(\dom
\Phi^*)^\perp$, so that
\[
 \mul \Phi^{**} (I+iB_\oplus ) \Phi^*
 =
 (\dom A_{1}^\half \cap \dom A_{2}^\half)^\perp.
\]
In particular, the form sum extension of $H_{1} + H_{2}$ or, equivalently, the closure of $\Phi$, is an operator
precisely when $\dom A_{1}^\half \cap \dom A_{2}^\half$ is dense in $\sH$.
The orthogonal operator part of $\Phi^{**} (I+iB_\oplus ) \Phi^*$ is the
maximal sectorial operator which corresponds to the form sum
\eqref{s-s-fs} restricted to the closure of $\dom H_{1}^\half \cap
\dom H_{2}^\half$.
As a comparison with $(H_1+H_2)_K$ recall that $\Phi^*\subset K^*$ by Lemma \ref{lem3.1}
and that $H_1+H_2$ is ``sectorially closable''  if and only if $(H_1+H_2)_K$ is an operator,
or, equivalently, $K^*$ is densely defined (see Section \ref{sec3.2}). In particular,
if the form sum is densely defined then also $(H_1+H_2)_K$ is a densely defined operator.


\bibliographystyle{amsplain}

\end{document}